\documentclass[12pt]{article}
\usepackage{amsfonts}
\usepackage{mathrsfs}
\usepackage[active]{srcltx}
\usepackage[curve]{xypic}
\usepackage{amsthm,amsfonts,amscd,amsmath,amssymb,color,array,verbatim,graphicx,caption2}
\input epsf.tex

\newtheorem{theorem}{Theorem}[section]
\newtheorem{lemma}[theorem]{Lemma}
\newtheorem{proposition}[theorem]{Proposition}
\newtheorem{corollary}[theorem]{Corollary}

\newtheorem*{theoremA}{Theorem A}
\newtheorem*{theoremB}{Theorem B}

\def\al{\alpha}
\def\be{\beta}
\def\g{\gamma}
\def\G{\Gamma}

\def\wh{\widehat}

\def\C{\mbox{$\mathbb C$}}

\def\FFF{{\mathcal F}}

\def\III{{\mathcal I}}

\def\JJJ{{\mathcal J}}
\def\LLL{{\mathcal L}}

\def\PPP{{\mathcal P}}

\def\UUU{{\mathcal U}}
\def\VVV{{\mathcal V}}

\def\sC{{\mathscr C}}
\def\sA{{\mathscr A}}

\def\cbar{\overline{\C}}
\def\smm{{\backslash}}

\begin{document}
\title{On the cycles of components \\ of disconnected Julia sets
}
\author{Guizhen Cui\thanks{The first author is supported by the NSFC Grant No. 11688101
and Key Research Program of Frontier Sciences, CAS, Grant No. QYZDJ-SSW-SYS005.}
\and Wenjuan Peng\thanks{The second author was supported by the NSFC Grant No. 11471317.}}
\date{\today}
\maketitle
\begin{abstract}
For any integers $d\ge 3$ and $n\ge 1$, we construct a hyperbolic rational map of degree $d$ such that it has $n$ cycles of the connected components of its Julia set except single points and Jordan curves.
\end{abstract}

\section{Introduction}
Let $f: \cbar\to\cbar$ be a rational map of the Riemann sphere with $\deg f\ge 2$. Denote by $\JJJ_f$ and $\FFF_f$ the Julia set and the Fatou set of $f$ respectively. Refer to \cite{Mi} for their definitions and basic properties. It is classical that $\JJJ_f$ is a non-empty compact set.

Assume that $\JJJ_f$ is disconnected. Let $K$ be a {\bf Julia component} of $f$, i.e., a connected component of $\JJJ_f$. Then each component of $f^{-1}(K)$ is also a Julia component. Thus $K$ is either {\bf periodic} if $f^p(K)=K$ for some integer $p\ge 1$, or {\bf eventually periodic} if $f^k(K)$ is periodic for some integer $k\ge 0$, or {\bf wandering} if $f^n(K)$ is disjoint from $f^m(K)$ for any integers $n>m\ge 0$.

If $K$ is periodic with period $p\ge 1$, then either $\deg(f^p|_K)=1$ and $K$ is a single point, or $\deg(f^p|_K)>1$ and there exists a rational map $g$ with $\deg g=\deg(f^p|_K)$ such that $(f^p, K)$ is quasi-conformally conjuagate to $(g, \JJJ_g)$ in a neighborhood of $K$ (see \cite{Mc1}). If $K$ is wandering and $f$ is a polynomial, then $K$ is a single point (refer to \cite{BH, KS, QY}).

The situation for general rational maps is more complicated. There are examples of rational maps whose wandering components of their Julia sets are Jordan curves \cite{Mc1}. In fact, a wandering Julia component of a geometrically finite rational map is either a single point or a Jordan curve \cite{PT1}.

A periodic Julia component $K$ is called {\bf simple} if either $K$ is a single point or each component of $f^{-n}(K)$ is a Jordan curve for any $n\ge 0$. It is called {\bf complex} otherwise. Denote
$$
N(f)=\#\{\text{cycles of complex periodic Julia components of } f\}.
$$
Refer to \cite{PT1} for the next theorem.

\begin{theoremA}
Let $f$ be a geometrically finite rational map with disconnected Julia set. Then $N(f)<\infty$ and each wandering Julia component of $f$ is either a single point or a Jordan curve.
\end{theoremA}

A natural problem is whether Theorem A holds for general rational maps. It is easy to see that $N(f)\le\deg f-2$ when $f$ is a polynomial. A related problem is whether $N(f)$ is bounded by a constant depending only upon $\deg f$ for rational maps. In this work, we will construct examples to show that this is not true.

\begin{theorem}\label{main}
Given any integers $d\ge 3$ and $n\ge 1$. There exists a hyperbolic rational map $g$ such that $\deg g=d$ and $N(g)=n$.
\end{theorem}

The main tools used in the proof of Theorem \ref{main} are canonical decompositions and Shishikura tree maps. The idea of canonical decompositions firstly appeared in \cite{CT}. The Shishikura tree maps were proposed by Shishikura in \cite{Shi2,Shi3}. In the following, we will explain how to apply them in this work.

In the early 1980s, Thurston established a complete topological characterization of postcritically finite rational maps by means of Thurston obstructions (\cite{DH1, Th}). A Thurston obstruction is a condition about the growth on weighted pullbacks of a collection of Jordan curves. Thurston's characterization theorem provides a criterion to find a rational realization of the given combinatorics. More precisely, in order to build a rational map with desired combinatorial properties, one may construct a topological branched covering, and then check whether this specific branched covering has Thurston obstructions or not. If not, then Thurston's theorem guarantees the existence of a rational map with the same combinatorial properties.

Thurston's theorem is a useful tool in holomorphic dynamics to construct various kind of rational maps with prescribed dynamical properties and there are many applications of Thurston's theory.  However, Thurston's theorem can only be applied to postcritically finite branched coverings,  and to verify the absence of Thurston onstructions, one usually need to check infinitely many collections of Jordan curves. Thus in general, it is a bit difficult to apply Thurston's theorem effectively.

Over the years, many methods have been developed to overcome these two drawbacks. We mention here a result (refer to \cite{CT} or $\S2, \mathrm{Theorem\ B}$ in this article) which will be used as a main tool in the present work. Thurston's theorem was extended to the set of non-post-critically finite rational maps or sub-hyperbolic rational maps in the paper \cite{CT}. In their work, instead of modifying the original proof of Thurston's theorem to fit their situation, they decomposed the dynamics of a sub-hyperbolic semi-rational map into post-critically finite parts and the gluing parts. By using the method of decompositions, they provided effective criteria for the absence of obstructions. However, in their paper, they did not state the theorems in the form of decompositions.

To prove Theorem \ref{main} in this article, we will develop the idea of decompositions (which will be called \textbf{canonical decompositions}) and present a precise statement about the canonical decompositions to adopt to our situation, see $\S3, \mathrm{Theorem\ \ref{decomposition}}$. Then we will give a criterion for the absence of Thurston obstructions in terms of canonical decompositions, see $\S3, \mathrm{Theorem\ \ref{obstruction}}$.  A \textbf{canonical multicurve} derives from a canonical decomposition. We will present the relationship between the complementary components of a canonical multicurve and the complex Julia components, refer to Lemma \ref{Y2C}.

Douady and Hubbard (\cite{DH2}) introduced Hubbard's trees to describe the dynamics of post-critically finite polynomials. Since then, there have been several attempts to use the tree structures to encode the dynamics of Julia sets, and now the idea of using trees are pervasive in the holomorphic dynamics. For example, for a rational map with Herman rings, a certain kind of tree and piecewise linear map on it were proposed by Shishikura (\cite{Shi2,Shi3}) to reflect the configuration of the Herman rings; DeMarco and McMullen (\cite{DM}) characterized the branched coverings of metrized, simplicial trees arising from polynomial maps with disconnected Julia sets, and gave a compactification of the moduli space of polynomial maps of degree at least 2; A family of cubic rational maps with buried Julia components coming from a particular tree was constructed by Godillon (\cite{G}).

To enumerate the cycles of complex Julia components, we will introduce a tree map associated with a canonical multicurve to characterize the dynamics on the configuration of Julia components (see $\S5$ and $\S6$). This idea is actually inspired by the tree maps proposed by Shishikura (\cite{Shi2, Shi3}).

For the purpose to construct rational maps with given number of cycles of complex Julia components, we establish a procedure on a tree map called \textbf{self-grafting} to create a new tree map, such that the corresponding new rational map has a new cycle of complex Julia components (refer to $\S7$ for the procedure). More precisely, starting from the original rational map and the Shishikura tree map associated with it, we will construct a rational map to realize the new Shishikura tree map with the same degree as the given rational map, and having one more cycle of complex Julia components, compared to the original rational map.  In order to realize the new Shishikura tree, we will first build a sub-hyperbolic semi-rational map and then apply Theorem \ref{obstruction} to verify the absence of Thurston obstructions for the specific branched covering (see $\S7, \mathrm{Theorem}\ \ref{self-grafting}$). This step is crucial to prove Theorem \ref{main}.

We would like to mention here an application of the idea in this work.  In \cite{AC}, based on the approaches of the canonical decompositions and self-graftings, a sequence of rational maps with an arbitrary large number of dynamically independent and non-monomial rescaling limits were constructed. From this result, they showed the existence of a family of rational maps with a non-trivial dynamics on the Berkovich projective line over the field of formal Puiseux series. Their work demonstrates the connections of the present work with arithmetic dynamics.

\section{Sub-hyperbolic version of Thurston's theorem}

In this section, we will first recall the sub-hyperbolic version of Thurston's theorem. Then we will present some preliminary lemmas about multicurves.

Let $F$ be a branched covering of $\cbar$ with $\deg F\ge 2$. Denote by $\Omega_F$ the set of branched points of $F$ and by
$$
\PPP_F=\overline{\bigcup_{n>0}F^n(\Omega_F)}
$$
{\bf the post-critical set} of $F$. The map $F$ is {\bf geometrically finite} if the accumulation point set $\PPP'_F$ of $\PPP_F$ is finite.

A geometrically finite branched covering $F$ is a {\bf (sub-hyperbolic) semi-rational map} if $F$ is holomorphic in a neighborhood of $\PPP'_F$ and each cycle in $\PPP'_F$ is either attracting or super-attracting.

Two semi-rational maps $F$ and $G$ are {\bf c-equivalent} if there exist a pair of orientation preserving homeomorphisms $(\phi,\psi)$ of $\cbar$ and an open set $U\supset\PPP'_F$ such that $G\circ\psi=\phi\circ F$, $\phi$ is holomorphic in $U$, $\psi=\phi$ in $U$ and $\psi$ is isotopic to $\phi$ rel $\PPP_F\cup U$.

Let $F$ be a semi-rational map. A Jordan curve $\g$ in $\cbar\smm\PPP_F$ is {\bf trivial} if one component of $\cbar\smm\g$ is disjoint from $\PPP_F$; or is {\bf peripheral} if one component of $\cbar\smm\g$ contains exactly one point of $\PPP_F$; or is {\bf essential} otherwise, i.e. if each component of $\cbar\smm\g$ contains at least two points of $\PPP_F$.

{\bf Convention}. For the simplicity of the writing, we say that two essential curves are isotopic if they are isotopic rel the post-critical set.

A {\bf multicurve}\ $\Gamma$ is a non-empty and finite collection of disjoint Jordan curves in $\cbar\smm\PPP_F$, each essential and no two isotopic. It is {\bf stable}\ if each essential curve in $F^{-1}(\g)$ for $\g\in\G$ is isotopic to a curve in $\G$; or {\bf pre-stable}\ if each curve $\g\in\G$ is isotopic to a curve in $F^{-1}(\be)$ for some $\be\in\G$. A multicurve is {\bf completely stable} if it is stable and pre-stable.

The transition matrix $M(\Gamma)=(a_{\beta\gamma})$ of a multicurve $\G$ is defined by the formula
$$
a_{\beta\gamma}=\sum_{\delta}\frac 1{\deg(F:\,\delta\to\gamma)}
$$
where the sum is taken over all components $\delta$ of $F^{-1}(\gamma)$ which are isotopic to $\beta$. Let $\lambda(\Gamma)=\lambda(M(\Gamma))$ denote the spectral radius of $M(\Gamma)$. A stable multicurve $\Gamma$ is called a {\bf Thurston obstruction} of $F$ if $\lambda(\Gamma)\ge 1$. Refer to \cite[Theorem 1.1]{CT} or \cite{JZ} for the next Theorem.

\begin{theoremB}
Let $F$ be a semi-rational map with $\PPP'_F\neq\emptyset$. Then $F$ is c-equivalent to a rational map $f$ if and only if it has no Thurston obstruction.
Moreover, the rational map $f$ is unique up to holomorphic conjugation.
\end{theoremB}

The following properties about multicurves will be used in this paper.

A multicurve $\G$ is called {\bf irreducible} if for each pair $(\g,\be)\in\G\times\G$, there exists a sequence $\{\g=\delta_0, \cdots, \delta_n=\beta\}$ of curves in $\G$ such that $F^{-1}(\delta_k)$ has a component isotopic to $\delta_{k-1}$ for $1\le k\le n$. Refer to \cite[Theorem B.6]{Mc2} for the next lemma.

\begin{lemma}\label{irreducible}
Let $F$ be a semi-rational map. For any multicurve $\G$ with $\lambda(\G)>0$, there is an irreducible multicurve $\G_0\subset\G$ such that $\lambda(\G_0)=\lambda(\G)$.
\end{lemma}

Refer to \cite[Corollary A.2]{CT} for the next lemma.

\begin{lemma}\label{eigenvalue}
For any non-negative square matrix $M$, its leading eigenvalue satisfies
$$
\lambda(M)=\inf\{\lambda:\, \exists\, v>0\text{ such that }Mv<\lambda v\}.
$$
\end{lemma}

\begin{lemma}\label{submulticurve}
Let $\G_1\subset\G$ be multicurves. Then $\lambda(\G_1)\le\lambda(\G)$.
\end{lemma}

\begin{proof} The transition matrices of $\G$ and $\G_1$ satisfy the following inequality:
$$
M(\G)=\left(\begin{array}{cc}
M(\G_1) & * \\
* & *
\end{array}\right)\ge M=
\left(\begin{array}{cc}
M(\G_1) & O_1 \\
O_2 & O_3
\end{array}\right),
$$
where $O_i$ are zero matrices. Thus $\lambda(\G)\ge\lambda(M)$ by \cite[Corollary A.3]{CT}.

By Lemma \ref{eigenvalue}, for any $\lambda>\lambda(M)$, there exists a vector $v=(v_1, v_2)>0$ such that $Mv<\lambda v$. Thus
$$
Mv=\left(\begin{array}{cc}
M(\G_1) & O_1 \\
O_2 & O_3
\end{array}\right)
\left(\begin{array}{c}
v_1 \\
v_2
\end{array}\right)=
\left(\begin{array}{c}
M(\G_1)v_1 \\
O
\end{array}\right)<
\left(\begin{array}{c}
\lambda v_1 \\
\lambda v_2
\end{array}\right),
$$
where $O$ is a zero matrix. So $M(\G_1)v_1<\lambda v_1$. Therefore $\lambda(\G_1)<\lambda$ by Lemma \ref{eigenvalue}. Since $\lambda$ is an arbitrary number with $\lambda>\lambda(M)$, we have $\lambda(\G_1)\le\lambda(M)$. Now the lemma follows.
\end{proof}

\begin{lemma}\label{block}
Let $\G=\G_1\sqcup\G_2$ be a multicurve of a semi-rational map $F$ such that for each curve $\g\in\G_2$, $F^{-1}(\g)$ has no component isotopic to a curve in $\G_1$. Then
$$
\lambda(\G)=\max\{\lambda(\G_1), \lambda(\G_2)\}.
$$
\end{lemma}

\begin{proof}
The transition matrix of $\G$ has the following block decomposition
$$
M(\G)=\left(\begin{array}{cc}
M(\G_1) & O \\
B & M(\G_2)
\end{array}\right),
$$
where $O$ is a zero matrix.

For any $\lambda>\max\{\lambda(\G_1), \lambda(\G_2)\}$, by Lemma \ref{eigenvalue}, there exist vectors $v_1, v_2>0$ such that $M(\G_1)v_1<\lambda v_1$ and $M(\G_2)v_2<\lambda v_2$. Thus there exists $\varepsilon>0$ such that $M(\G_2)v_2+\varepsilon Bv_1<\lambda v_2$. Now we have
$$
M(\G)\left(\begin{array}{c}
\varepsilon v_1 \\
v_2
\end{array}\right)
=\left(\begin{array}{c}
\varepsilon M(\G_1)v_1 \\
M(\G_2)v_2+\varepsilon Bv_1
\end{array}\right)
<\left(\begin{array}{c}
\varepsilon\lambda v_1 \\
\lambda v_2
\end{array}\right).
$$
Thus $\lambda(\G)<\lambda$ by Lemma \ref{eigenvalue}. So $\lambda(\G)\le\max\{\lambda(\G_1), \lambda(\G_2)\}$ since $\lambda$ is arbitrary.

By Lemma \ref{submulticurve}, we have $\lambda(\G_1)\le\lambda(\G)$ and $\lambda(\G_2)\le\lambda(\G)$. Combining these inequalities, we obtain $\lambda(\G)=\max\{\lambda(\G_1), \lambda(\G_2)\}$.
\end{proof}

\section{Canonical decompositions and canonical multicurves}
In this section, we will introduce the canonical decompositions in \cite{CT} and give a criterion for the absence of Thurston obstructions in terms of decompositions.

By quasi-conformal surgery, for any sub-hyperbolic rational map $g$, there exist another sub-hyperbolic rational map $f$ and a quasi-conformal map $\phi$ of $\cbar$ such that $f\circ\phi=\phi\circ g$ holds in a neighborhood of $\JJJ_g$, and each super-attracting cycle of $f$ is contained in $\PPP'_f$. This can be done. In fact, one could perturb all the super-attracting cycles into attracting cycles.

The main object of this work is Julia sets. So we may assume that sub-hyperbolic rational maps $f$ in our consideration satisfy the condition that each super-attracting cycle of $f$ is contained in $\PPP'_f$. This technical assumption will simplify the statements and proofs in this work.

Let $F$ be semi-rational map. It is called {\bf generic} if $\PPP'_F\neq\emptyset$ and each super-attracting cycle of $F$ is contained in $\PPP'_F$.

A set $E\subset\cbar$ is {\bf D-type} if there exists a simply-connected domain $D\subset\cbar$ such that $E\subset D$ and $D$ contains at most one point of $\PPP_F$; or is {\bf A-type} if it is not D-type and there exists an annulus $A\subset\cbar$ such that $E\subset A$ and $A$ is disjoint from $\PPP_F$; or is {\bf Q-type} otherwise. We remark that the definitions of D-type, A-type and Q-type sets depend on the perturbation of a semi-rational map. For instance, for a sub-hyperbolic rational map, a super-attracting Fatou domain containing exactly one $\PPP_f$ point is D-type, whereas it is Q-type under our assumption that each super-attracting cycle of $f$ is contained in $\PPP'_f$. We make the assumption to simplify the discussion.

From the condition $F(\PPP_F)\subset\PPP_F$, we know that $F(E)$ is Q-type if $E$ is Q-type, and $F(E)$ is A-type or Q-type if $E$ is A-type.

A generic semi-rational map $F$ is called {\bf degenerate} if for any open set $U\supset\PPP'_F$, there exists $N>0$ such that for $n>N$, each component of $\cbar\smm F^{-n}(U)$ is D-type.

\begin{proposition}\label{degenerate}
A degenerate and generic semi-rational map is always c-equivalent to a rational map.
\end{proposition}

\begin{proof}
Let $F$ be a degenerate and generic semi-rational map. Let $\G$ be a multicurve with $\lambda(\G)>0$. Then there is an irreducible multicurve $\G_0\subset\G$ such that $\lambda(\G_0)=\lambda(\G)$ by Lemma \ref{irreducible}. Pick an open set $U\supset\PPP'_F$ such that it is disjoint from all the curves in $\G_0$. Then there exists $N>0$ such that for $n>N$, each component of $\cbar\smm F^{-n}(U)$ is D-type. It follows that for each $\g\in\G_0$, any curve in $F^{-n}(\g)$ is contained in a D-type set and hence is non-essential. This contradicts the condition that $\G_0$ is irreducible. Thus $F$ has no Thurston obstruction and hence it is c-equivalent to a rational map by Theorem B.
\end{proof}

We will show in Corollary \ref{cantor set} that if $F$ is a degenerate and generic sub-hyperbolic rational map, its Julia set is a Cantor set.

By a {\bf tame} set $U\subset\cbar$, we mean that $U$ is open and has only finitely many components whose closures are pairwise disjoint, and each component is bounded by finitely many Jordan curves.

\begin{theorem}\label{decomposition}
Let $F$ be a non-degenerate and generic semi-rational map. There exist a completely stable multicurve $\G$ and a tame set $\UUU_0\subset\cbar$ with $\PPP'_F\subset\UUU_0$ and $\partial\UUU_0\cap\PPP_F=\emptyset$, such that the following conditions hold for $n\ge 0$. Denote $\UUU_n=F^{-n}(\UUU_0)$
and $\LLL_n=\cbar\smm\UUU_n$.

(a) $\UUU_0$ is compactly contained in $\UUU_1$ (denote $\UUU_0\Subset\UUU_1$).

(b) $\{F^k(\UUU_0)\}$ converges to $\PPP'_F$ as $k\to\infty$.

(c) Each essential curve on $\partial\UUU_n$ is isotopic to a curve in $\G$, and vice versa.

(d) Each Q-type component $D_{n+1}$ of $\UUU_{n+1}$ contains exactly one Q-type component $D_n$ of $\UUU_n$, and each component of $D_{n+1}\smm\overline{D_n}$ is not Q-type.

(e) Each Q-type component of $\LLL_{n}$ contains exactly one Q-type component of $\LLL_{n+1}$.
\end{theorem}

\begin{proof}
Pick Koenigs or B\"{o}ttcher disks at every cycles in $\PPP'_F$ such that their boundaries are disjoint from $\PPP_F$. Denote their union by $V_0$. Then $V_0\Subset F^{-1}(V_0)$ and $\{F^k(V_0)\}$ converges to $\PPP'_F$ as $k\to\infty$. Denote $V_n=F^{-n}(V_0)$ for $n\ge 1$. Then $V_{n-1}\Subset V_n$.

Note that components of $\partial V_n$ are pairwise disjoint for all $n\ge 0$. There exists a multicurve $\Lambda_n$ such that each curve in $\Lambda_n$ is contained in $\bigcup_{i=0}^n\partial V_i$ and each essential curve on $\bigcup_{i=0}^n\partial V_i$ is isotopic to a curve in $\Lambda_n$. It follows that each curve in $\Lambda_n$ is isotopic to a curve in $\Lambda_{n+1}$. In particular, $\#\Lambda_n$ is increasing.

Denote by $m$ the number of components of $V_0\cup\PPP_F$. Then $\#\Lambda_n\le m-3$ for all $n\ge 0$ since each curve in $\Lambda_n$ is disjoint from $V_0\cup\PPP_F$. So there exists $N\ge 0$ such that $\#\Lambda_n$ is a constant for $n\ge N$. Thus there is a multicurve $\Lambda$ such that for any $n\ge 0$, each essential curve on $\partial V_n$ is isotopic to a curve in $\Lambda$.

Define a sub-multicurve $\G\subset\Lambda$ by $\g\in\G$ if for any $N>0$, there exists a curve $\beta\subset\partial V_n$ with $n>N$ such that $\beta$ is isotopic to $\g$. It is well-defined since $\G$ is non-empty by the condition that $F$ is non-degenerate. By the definition of $\G$, there exists $n_1\ge 0$ such that each essential curve on $\partial V_n$ with $n\ge n_1$ is isotopic to a curve in $\G$.

\begin{figure}[htbp]
\begin{center}
\includegraphics[width=11cm]{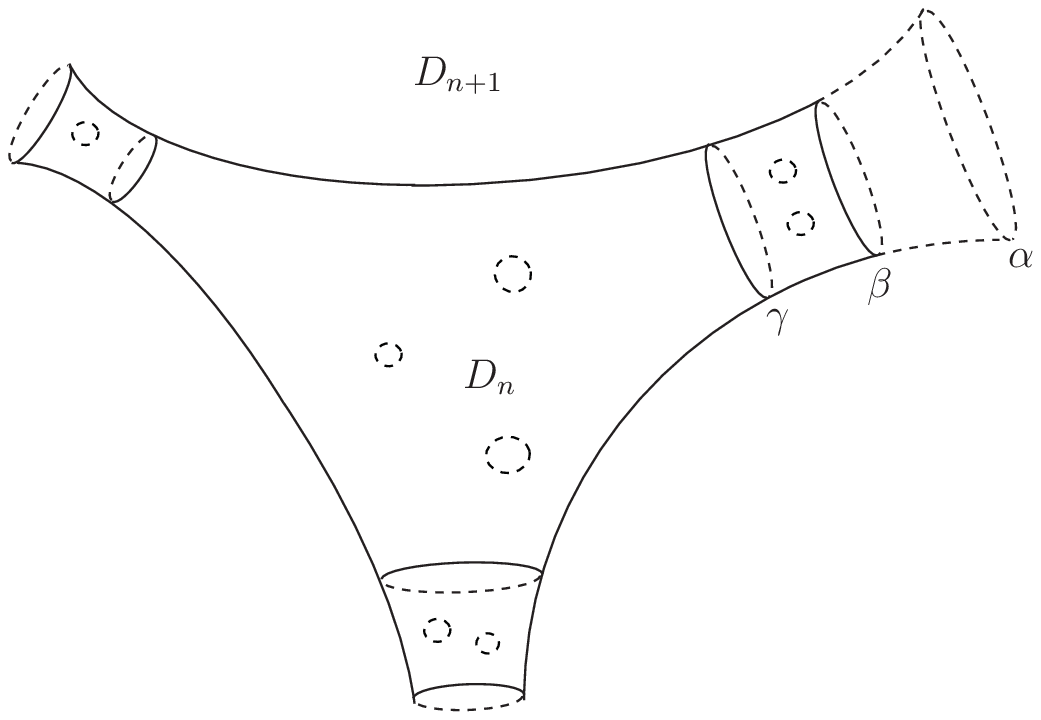}
\end{center}
\begin{center}{\sf Figure 1. A Q-type component of $V_n$}
\end{center}
\end{figure}

Let $\g\subset\partial V_n$ be an essential curve with $n\ge n_1$. It is isotopic to a curve in $\G$. By the definition of $\G$, it is also isotopic to a curve $\alpha\subset\partial V_N$ for some integer $N>n+1$. Let $D_{n+1}$ be the component of $V_{n+1}$ that contains $\g$. Since $\overline{D_{n+1}}$ is disjoint from $\alpha$, there exists a unique curve $\beta\subset\partial D_{n+1}$ such that $\beta$ separates $\g$ from $\alpha$. Thus $\beta$ is isotopic to $\g$ since $\g$ is isotopic to $\alpha$. Refer to Figure 1.

For $n\ge n_1$, let $\G_n\subset\G$ be the multicurve defined by $\g\in\G_n$ if $\g$ is isotopic to a curve on $\partial V_n$. The above discussion shows that $\G_n\subset\G_{n+1}$. Thus there exists $n_2\ge n_1$ such that $\G_n=\G_{n_2}$ for $n\ge n_2$. By the definition of $\G$, we have $\G_n=\G$ for $n\ge n_2$.

Let $D_n$ be a Q-type component of $V_n$ with $n\ge n_2$. Let $D_{n+1}$ be the component of $V_{n+1}$ that contains $D_n$. Then $D_{n+1}$ is also Q-type.
Let $B$ be a component of $D_{n+1}\smm D_n$. Then $\g=\partial B\cap\partial D_n$ is a curve. If $\g$ is non-essential, then $B$ is D-type since $D_n$ is Q-type. If $\g$ is essential, then it is isotopic to a curve $\beta\subset\partial D_{n+1}$ by the above discussion. Since $D_n$ is Q-type, $B$ is contained in the closed annulus bounded by $\g$ and $\beta$. Thus $B$ is not Q-type. It implies that $D_{n+1}$ contains exactly one Q-type component of $V_n$. See Figure 1.

Let $s(n)$ be the number of Q-type components of $V_n$ and $r(n)$ be the number of Q-type components of $\cbar\smm V_n$. Then $s(n)+r(n)=\#\G+1$ and $s(n)$ is increasing for $n\ge n_2$. So there exists $n_3\ge n_2$ such that $s(n)$ is a constant for $n\ge n_3$. This shows that each Q-type component of $V_{n+1}$ contains exactly one Q-type component of $V_n$, and each Q-type component of $\cbar\smm V_{n}$ contains exactly one Q-type component of $\cbar\smm V_{n+1}$.

Set $\UUU_0=V_n$ for some $n\ge n_3$. It satisfies the conditions of the theorem.
\end{proof}

We will call $\G$ a {\bf canonical multicurve} of $F$ and $(\UUU_0, \LLL_0)$ a {\bf canonical decomposition} of $F$, if they satisfy the conditions of Theorem \ref{decomposition}. Canonical decompositions and canonical multicurves are uniquely determined by semi-rational maps in the sense of homotopy by the next proposition.

\begin{proposition}\label{canonical}
Let $(\UUU_0, \LLL_0)$ be a canonical decomposition of $F$.
Let $V_0\Subset\cbar$ be a tame set such that $\PPP'_F\subset V_0$ and $\partial V_0$ is disjoint from $\PPP_F$. Suppose that $V_0\Subset F^{-1}(V_0)$ and $\{F^k(V_0)\}$ converges to $\PPP'_F$ as $k\to\infty$. Set $V_n=F^{-n}(V_0)$ and $\UUU_n=F^{-n}(\UUU_0)$ for $n\ge 0$. Then there exists $N\ge 0$ such that $(V_n, \cbar\smm V_n)$ is a canonical decomposition of $F$ for $n\ge N$. Moreover, the following conditions hold for $n\ge N$.

(a) Each essential curve on $\partial V_n$ is isotopic to an essential curve on $\partial\UUU_0$, and vice versa.

(b) There exist integers $0<i<j$ such that $\UUU_i\subset V_n\subset\UUU_j$. Each Q-type component of $\UUU_j$ contains exactly one Q-type component of $V_n$ and each Q-type component of $V_n$ contains exactly one Q-type component of $\UUU_i$.
\end{proposition}

Before proving the above proposition, we want to point out the following facts. It is easy to check that for $n\ge 1$,  $(F^{-n}(\UUU_0), F^{-n}(\LLL_0))$ is also a canonical decomposition of $F$, and $(\UUU_0, \LLL_0)$ is a canonical decomposition of $F^n$. Let $\UUU'_0$ be the union of A-type or Q-type components of $\UUU_0$, then $(\UUU'_0, \cbar\smm\UUU'_0)$ is also a canonical decomposition of $F$. Moreover, each A-type component of $\cbar\smm\UUU'_0$ is a closed annulus. In the following, we will prove Proposition \ref{canonical}.

\begin{proof}
Both $\{F^k(V_0)\}$ and $\{F^k(\UUU_0)\}$ converge to $\PPP'_F$ as $k\to\infty$. So there exist integers $p, q>0$ such that $\UUU_0\Subset V_p\Subset \UUU_{p+q}$. Thus for $n\ge p$,
$$
\UUU_{n-p}\Subset V_{n}\Subset\UUU_{n+q}.
$$

(a) Let $D$ be an A-type or a Q-type component of $V_n$ for some $n\ge p$. Then $D$ is contained in a component $U$ of $\UUU_{n+q}$, which is either an A-type or a Q-type. If $U$ is A-type, then each essential curve on $\partial D$ is isotopic to a curve on $\partial U$. Thus it is also isotopic to a curve on $\partial\UUU_{0}$ since $\UUU_0$ is a canonical decomposition.

Suppose that $U$ is Q-type. Then there is a unique Q-type component $U'$ of $\UUU_{n-p}$ such that $U'\subset U$. Moreover, each component of $U\smm\overline{U'}$ is not Q-type. If $D\subset U\smm U'$, then each essential curve on $\partial D$ is isotopic to a curve on $\partial U$. Thus it is also isotopic to a curve on $\partial\UUU_{0}$.

If $D\cap U'\neq\emptyset$, then $U'\subset D$. For each essential curve $\g$ on $\partial D$, there is a unique curve $\beta$ on $\partial U'$ such that it separates $U'$ from $\g$. On the other hand, $\beta$ is isotopic to a curve $\alpha$ on $\partial\UUU_{n+q}$. Since $U'$ is Q-type, $\beta$ must separate $U'$ from $\alpha$. This implies that both $\g$ and $\alpha$ are contained in the same complementary component of $\beta$. Thus $\g$ is isotopic to $\beta$ and hence is isotopic to a curve on $\partial\UUU_{0}$.

Conversely, for each essential curve $\g$ on $\partial\UUU_0$, it is isotopic to a curve $\beta\subset\partial\UUU_{n-p}$ and a curve $\beta'\subset\partial \UUU_{n+q}$. Let $D$ be the component of $V_n$ that contains $\beta$. Since $\beta'$ is disjoint from $D$, there is a unique curve $\alpha$ on $\partial D$ such that $\alpha$ separates $\beta$ from $\beta'$. Thus $\alpha$ is isotopic to $\beta$ and hence is isotopic to $\g$.

(b) Let $B$ be a Q-type component of $\UUU_{n+q}$. Then there is a unique Q-type component $B'$ of $\UUU_{n-p}$ such that $B'\subset B$. Let $D$ be the component of $V_n$ that contains $B'$, then $D$ is Q-type and $D\subset B$. Obviously, $D$ contains exactly one Q-type component of $\UUU_{n-p}$. Since each component of $B\smm\overline{B'}$ is not Q-type, $D$ is the unique Q-type component of $V_n$ contained in $B$. This proves (b).

Let $D_{n+1}$ be a Q-type component of $V_{n+1}$. Then $D_{n+1}$ contains exactly one Q-type component of $\UUU_{n-p}$. Thus $D_{n+1}$ contains exactly one Q-type component of $V_n$ since each Q-type component of $V_n$ also contains exactly one Q-type component of $\UUU_{n-p}$. Let $B$ be a component of $D_{n+1}\smm\overline{D_n}$. Then $B$ is contained in a component of $\UUU_{n+q+1}\smm\overline{\UUU_{n-p}}$. By Theorem \ref{decomposition} (d),
each component of $\UUU_{k+1}\smm\overline{\UUU_k}$ is not Q-type for $k>1$. It deduces that each component of $\UUU_{k+l}\smm\overline{\UUU_k}$ is not Q-type for any integer $l\ge 1$. Therefore $B$ is not Q-type. It concludes that $(V_n, \cbar\smm V_n)$ is a canonical decomposition of $F$. Now the proof is complete.
\end{proof}

Let $(\UUU, \LLL)$ be a canonical decomposition of $F$. From Theorem \ref{decomposition}, we may define a map $\chi_F$ on the collection of Q-type components of $\LLL$ by $\chi_F(L_i)=L_j$ if the unique Q-type component of $F^{-1}(\LLL)$ in $L_i$ maps to $L_j$ by $F$. Since this collection is finite, each Q-type component of $\LLL$ is eventually periodic under $\chi_F$. We will call a Q-type component $L$ of $\LLL$ is {\bf periodic} if $L$ is periodic under $\chi_F$.

Obviously, for each periodic Q-type component $L$ of $\LLL$, $F^{-1}(L)$ has exactly one Q-type component contained in periodic Q-type components of $\LLL$.

Denote by $m\ge 1$ the total number of Q-type components of $\LLL$. Then for each non-periodic Q-type component $L$ of $\LLL$, each component of $F^{-m}(L)$ is not Q-type. Otherwise, assume that $L^m$ is a Q-type component of $F^{-m}(L)$, then $F^k(L^m)$ is also Q-type for $0\le k\le m$. Thus at least two of them, denoted by $L^i$ and $L^j$ with $i<j$, are contained in the same component of $\LLL$. So $L^j\subset L^i$. This implies that $L$ is periodic and hence is a contradiction. Now we have proved the next lemma.

\begin{lemma}\label{L}
(1) Each Q-type component $L$ of $\LLL$ is eventually periodic under $\chi_F$.

(2) For each periodic Q-type component $L$ of $\LLL$, $F^{-1}(L)$ has exactly one Q-type component contained in periodic Q-type components of $\LLL$.

(3) There exists $N\ge 1$ such that for each non-periodic Q-type component $L$ of $\LLL$ and any $n\ge N$, each component of $F^{-n}(L)$ is not Q-type.
\end{lemma}

Let $L$ be a Q-type component of $\LLL$. We say that a multicurve $\G$ of $F$ is {\bf essentially} contained in $L$ if for each curve $\g\in\G$, $\g\subset L$ and $\g$ is not isotopic to a curve on $\partial L$.

For a multicurve $\G$ of $F$ and an integer $p\ge 1$, we denote by $\lambda(\G, F^p)$ the leading eigenvalue of the transition matrix of $\G$ under $F^p$.

\begin{theorem}\label{obstruction}
Let $F$ be a non-degenerate and generic semi-rational map. Let $(\UUU, \LLL)$ be a canonical decomposition of $F$ and $\G_F$ be a canonical multicurve of $F$. Then $F$ is c-equivalent to a rational map if and only if $\lambda(\G_F)<1$ and for each periodic Q-type component $L$ of $\LLL$ with period $p\ge 1$ and any multicurve $\G$ contained essentially in $L$ , we have $\lambda(\G, F^p)<1$.
\end{theorem}

\begin{proof}
The necessity follows directly from Theorem B. In the following, we prove the sufficiency.

Let $\G_1$ be a multicurve of $F$ with $\lambda(\G_1)>0$. Then there is an irreducible multicurve $\G_0\subset\G_1$ such that $\lambda(\G_0)=\lambda(\G_1)$ by Lemma \ref{irreducible}.

There exists $n_0\ge 0$ such that $F^{n_0}(\UUU)$ is disjoint from $\G_0$. Thus for each $\g\in\G_0$, $F^{-n_0}(\g)$ is contained in $\LLL$. Since $\G_0$ is irreducible, we may choose a multicurve $\G'$ in $\LLL$ such that each curve in $\G'$ is isotopic to a curve in $\G_0$, and vice versa. Thus $\lambda(\G')=\lambda(\G_0)=\lambda(\G_1)$.

Since $\G_F$ is stable and $\G'$ is irreducible, either each curve in $\G'$ is isotopic to a curve in $\G_F$, or every curve in $\G'$ is not isotopic to a curve in $\G_F$. In the former case, $\G'$ is contained in $\G_F$ in the sense of isotopy and hence $\lambda(\G')<\lambda(\G_F)<1$ by Lemma \ref{submulticurve}.

Now we suppose that every curve in $\G'$ is not isotopic to a curve in $\G_F$. Then each curve in $\G'$ is contained in Q-type components of $\LLL$.

If a curve $\g\in\G'$ is contained in a non-periodic Q-type component $L$ of $\LLL$, then as $n$ is large enough, each component of $F^{-n}(\g)$ is contained in a D-type or an A-type component of $F^{-n}(\LLL)$ by Lemma \ref{L} (3). This contradicts the condition that $\G'$ is irreducible. Thus each curve $\g\in\G'$ is contained in a periodic Q-type component of $\LLL$.

Let $L_0$ be a periodic Q-type component of $\LLL$ with period $p\ge 1$ such that it contains curves of $\G'$. Denote $L_i=\chi_F(L_0)$ for $1\le i\le p$.
Then $L_{p}=L_0$.  Let $\Lambda_i\subset\G'$ be the sub-multicurve contained in $L_i$. Since $\G'$ is irreducible, we know that for $0\le i<p$, each curve $\beta\in\Lambda_{i}$ is isotopic to a curve in $F^{-1}(\g)$ for some $\g\in\Lambda_{i+1}$ by Lemma \ref{L} (2). Conversely, if $\g\in\G'\smm\Lambda_{i+1}$, then $F^{-1}(\g)$ has no component isotopic to a curve in $\Lambda_i$. This implies that $\G'=\bigcup_{i=0}^{p-1}\Lambda_i$ and
$$
M(\G')^p=\left(
\begin{array}{ccc}
M_0 & \cdots & O \\
\vdots & & \vdots \\
O & \cdots & M_{p-1}
\end{array}\right),
$$
where $M_i=M(\Lambda_i, F^p)$. By the condition of the theorem, $\lambda(M_i)<1$ for $0\le i<p$. Thus by Lemma \ref{block}, we have
$$
\lambda(\G')^p=\max\{\lambda(M_0),\cdots, \lambda(M_{p-1})\}<1.
$$
Therefore $F$ is c-equivalent to a rational map by Theorem B.
\end{proof}

\section{Complex components of Julia sets}
We will make use of canonical decompositions to characterize the configuration of complex Julia components in this section.

Let $f$ be a generic sub-hyperbolic rational map. Recall that a periodic Julia component $K$ of $f$ is simple if either $K$ is a single point or a Jordan curve disjoint from $\PPP_f$. It is a complex Julia component otherwise.

\begin{lemma}\label{AQ}
Let $K$ be a Julia component which is not a single point.

(1) If $K$ is wandering, then $f^n(K)$ is A-type as $n$ is large enough.

(2) If $K$ is a periodic Jordan curve disjoint from $\PPP_f$, then $K$ is A-type.

(3) If $K$ is a complex periodic Julia component, then $K$ is Q-type.
\end{lemma}

\begin{proof} Denote $K_n=f^n(K)$ for $n\ge 0$. Let $V$ be the union of all the periodic Fatou domains of $f$. Since $f$ is generic, $V$ is non-empty and each component of $V$ contains points of $\PPP'_f$.

(1) If $K$ is wandering, then $f^n(K)$ is a Jordan curve for all $n\ge 0$ by Theorem A. Assume by contradiction that $K_n$ is D-type for all $n\ge 0$.  Then there is exactly one complementary component $U_n$ of $K_n$ such that $V\subset U_n$. Denote $\wh K_n=\cbar\smm U_n$. Then $K_n\subset\wh K_n$ and $\wh K_n$ contains at most one point of $\PPP_f$. So $f(\wh K_n)=\wh K_{n+1}$. This shows that the forward orbit of $\wh K$ is always disjoint from $V$. Thus the interior of $\wh K$ is contained in the Fatou set. This contradicts the fact that the forward orbit of $\wh K$ is disjoint from $V$.

(2) Suppose that $K$ is a periodic Jordan curve disjoint from $\PPP_f$. Then $K$ is either A-type or D-type. Assume by contradiction that $K$ is D-type. Then $K_n$ is also D-type for all $n\ge 0$. Using the same argument as above, we could deduce again a contradiction. Thus $K$ is A-type.

(3) Suppose that $K$ is a complex periodic Julia component with period $p\ge 1$. If $K$ is D-type, then there is a Jordan domain $D\supset K$ such that $D$ contains at most one point of $\PPP_f$. Denote by $A$ the unique annulus component of $D\smm K$. We may require that $A$ is disjoint from $\PPP_f$. Denote by $\wh K=D\smm A$. Then $K\subset\wh K$ and $\partial\wh K\subset K$.

Let $D_1$ be the component of $f^{-p}(D)$ that contains $K$. Then $D_1$ is also a Jordan domain. Since $f$ is sub-hyperbolic, it is expanding in a neighborhood of $\JJJ_f$ under a degenerate metric. Thus we may choose $D$ such that $D_1\Subset D$. Let $D_n$ be the component of $f^{-np}(D)$ that contains $K$ for $n\ge 1$. Then $\wh K=\cap_{n=1}^{\infty} D_n$ and hence $f^p(\wh K)=\wh K$.

If $\deg(f^p|_{D_n})=1$, then $\wh K$ is a single point by Schwarz Lemma. This is a contradiction.

If $\deg(f^p|_{D_n})>1$, then $f$ has a unique critical value $a\in D$ since $D$ contains at most one point of $\PPP_f$. Moreover $a\in\wh K$ since $A$ is disjoint from $\PPP_f$. Thus $f^p(a)=a$ and hence the point $a$ is also the unique critical point of $f$ in $D$. So $a$ is a super-attracting point. This contradicts the assumption that $f$ is generic.

If $K$ is A-type, then $K$ is disjoint from $\PPP_f$ and there are exactly two components of $\cbar\smm K$ containing points of $\PPP_f$. Thus there is an annulus $A\subset\cbar$ such that $K\subset A$ and $A$ is disjoint from $\PPP_f$.
As above, we may choose $A$ such that $A_1\Subset A$, where $A_1$ is the component of $f^{-p}(A)$ that contains $K$. Thus $K$ is also a Jordan curve by a folklore argument. So $K$ is a simple Julia component. This is a contradiction. In conclusion, $K$ is Q-type.
\end{proof}

\begin{corollary}{\label{cantor set}}
The Julia set of a degenerate and generic sub-hyperbolic rational map is a Cantor set.
\end{corollary}

\begin{proof}
Let $f$ be a degenerate and generic sub-hyperbolic rational map. By definition, every Julia component of $f$ is D-type. So each Julia component of $f$ is a single point by the above lemma.
\end{proof}

\vskip 0.24cm
Q-type Julia components or Fatou domains are closely related to canonical decompositions. Let $f$ be a non-degenerate and generic sub-hyperbolic rational map. Let $(\UUU, \LLL)$ be a canonical decomposition of $f$ and $\G_f$ be a canonical multicurve of $f$ consisting of curves on $\partial\UUU$.

\begin{lemma}\label{Y2C}

(1) Each Q-type Fatou domain contains exactly one Q-type component of $\UUU$.

(2) Let $E$ be a component of $\cbar\smm\G_f$. Then exactly one of the following two possibilities occurs. Either $E$ contains exactly one Q-type Julia component, or $E$ contains exactly one Q-type component of $\UUU$.
\end{lemma}

\begin{proof}
(1) Let $D$ be a Q-type Fatou domain. Let $k\ge 0$ be an integer such that $f^k(D)$ is periodic. Note that $f^k(D)$ has only finitely many complementary components containing points of $\PPP_f$. Denote their union by $E$. Then there exists a domain $V\Subset f^k(D)$ bounded by finitely many pairwise disjoint Jordan curves, such that $V\cap\PPP_f$=$f^k(D)\cap\PPP_f$ and each complementary component of $V$ contains at most one component of $E$. In other words, each component $\Omega$ of $f^k(D)\smm\overline{V}$ has at most one complementary component $E_1$ containing points of $\PPP_f$ except the complementary component $E_0$ that contains $V$. Consequently, $V$ is Q-type.

Now $B:=\cbar\smm (E_0\cup E_1)\supset\Omega$ is an annulus disjoint from $\PPP_f$. Thus each component of $f^{-n}(B)$ for $n\ge 1$ is also an annulus disjoint from $\PPP_f$. As a consequence, for each component $\Omega'$ of $f^{-k}(\Omega)$, $\Omega'$ has at most two complementary components containing points of $\PPP_f$ and $\partial\Omega'$ has exactly one component intersecting with $f^{-k}(\overline{V})$. Thus $f^{-k}(V)$ has a unique component $V'\Subset D$ and for any domain $K\Subset D\smm\overline{V'}$, $K$ is either D-type or A-type. Since $D$ is Q-type, $V'$ is also Q-type.

There exists an integer $n>0$ such that $V'\subset f^{-n}(\UUU)$. Let $U_n$ denote the component of $f^{-n}(\UUU)$ that contains $V'$. Then $U_n$ is Q-type,  $U_n\Subset D$ and for any domain $K\Subset D\smm\overline{U_n}$, $K$ is either D-type or A-type.

Let $U_0$ be the unique Q-type component of $\UUU$ with $U_0\subset U_n$. Then $U_0\Subset D$ and $D\smm U_0$ contains no other Q-type components of $\UUU$. Thus $D$ contains a unique Q-type component of $\UUU$.

(2) Let $E$ be a component of $\cbar\smm\G_f$. Then $E$ is Q-type since each curve in $\G_f$ is essential. We claim that it contains at most one Q-type Julia component or one component of $\UUU$.

Obviously, $E$ contains at most one Q-type component of $\UUU$. Assume by contradiction that $E$ contains two Q-type Julia components $K_1$ and $K_2$, then there exists an A-type or a Q-type Fatou domain $B\subset E$ separating $K_1$ from $K_2$.

Let $B_n$ be the union of components of $f^{-n}(\UUU)$ contained in $B$. Then $B_n\Subset B_{n+1}$ and $\bigcup_{n\ge 1} B_n=B$ since $\bigcup_{n\ge 1} f^{-n}(\UUU)=\FFF_f$. Thus there is an integer $n>0$ such that $B_n$ separates $K_1$ from $K_2$. Therefore there exists a curve $\beta\subset\partial f^{-n}(\UUU)$ which separates $K_1$ from $K_2$. The curve $\beta$ is not isotopic to any curve in $\G_f$ since both $K_1$ and $K_2$ are Q-type. This is a contradiction.

If $E$ contains a Q-type component $U$ of $\UUU$ and a Q-type Julia component $K$, then there exists a curve $\beta\subset\partial U$ such that $\beta$ is isotopic to a curve in $\Gamma_f$ and separates $U$ from $K$. This is also a contradiction. Now the claim is proved.

Assume now that $E$ contains no Q-type components of $\UUU$, then its closure must contain a Q-type component $L$ of $\LLL$. By Theorem \ref{decomposition} (e), $f^{-1}(\LLL)$ has exactly one Q-type component $L_1$ such that $L_1\subset L$. Inductively, $f^{-n}(\LLL)$ has exactly one Q-type component $L_n$ such that $L_n\subset L_{n-1}$ for $n\ge 2$. Set $K=\cap_{n=0}^{\infty} L_n$. Then $K$ is a Q-type continuum which is disjoint from $\bigcup_{n\ge 1}f^{-n}(\UUU)=\FFF_f$. Therefore $K\subset E$ is a Q-type Julia component.
\end{proof}

\section{The Shishikura tree map}
In this section, we will introduce the Shishikura tree map associated with the canonical multicurve of a semi-rational map, and show the relation between the transition matrix of the canonical multicurve and the Shishikura tree map. We will also discuss the dynamics of a tree map.

By a {\bf tree map} we mean a finite tree $T$ with the vertex set $X\subset T$ and a continuous map $\tau: T\to T$ such that $\tau^{-1}(X)\supset X$ is a finite set and $\tau$ is linear on components of $T\smm\tau^{-1}(X)$ under some linear metric on $T$.

Let $F$ be a non-degenerate and generic semi-rational map. Let $\G$ be a canonical multicurve of $F$. We want to construct a tree map associated with $\G$ such that it characterizes the configuration and the dynamics of the components of $\cbar\smm\G$.

The {\bf Shishikura tree} $(T_F, X_0)$ of $F$ is the dual tree of $\G$ defined as the following. There exists a bijection $v$ from the collection of components of $\cbar\smm\G$ to the vertex set $X_0$. For two distinct components $E_1, E_2$ of $\cbar\smm\G$, the vertices $v(E_1)$ and $v(E_2)$ are connected by an edge of $(T_F, X_0)$ if $E_1$ and $E_2$ have a common boundary component, which is a curve in $\G$. Thus there is a bijection $e$ from $\G$ to the collection of edges of $T_F$.

By the definition, the bijection $v$ is {\bf order-preserving}, i.e., for distinct components $E_0, E_1$ and $E_2$ of $\cbar\smm\G$, $E_0$ separates $E_1$ from $E_2$ if and only if $v(E_0)$ separates $v(E_1)$ from $v(E_2)$ in $T_F$.

Let $\G_1$ be the collection of essential curves in $F^{-1}(\G)$. Each component $E_1$ of $\cbar\smm\G_1$ is either A-type or Q-type. In the latter case, $E_1$ is {\bf isotopic} to a component $E$ of $\cbar\smm\G$, i.e., there exists a homeomorphism $\theta:\cbar\to\cbar$ isotopic to the identity rel $\PPP_F$ such that $\theta(E_1)=E$.

There exists an injection $v_1$ from the collection of components of $\cbar\smm\G_1$ into $T_F$ such that it satisfies the following conditions:

(1) $v_1(E_1)=v(E)$ if $E_1$ is Q-type, where $E$ is the component of $\cbar\smm\G$ isotopic to $E_1$.

(2) $v_1(E_1)\in e(\g)$ if $E_1$ is A-type, where $\g\in\G$ is isotopic to a curve on $\partial E_1$.

(3) $v_1$ is order-preserving, i.e. for distinct components $E_0, E_1$ and $E_2$ of $\cbar\smm\G_1$, $E_0$ separates $E_1$ from $E_2$ if and only if $v_1(E_0)$ separates $v_1(E_1)$ from $v_1(E_2)$ in $T_F$.

Denote by $X_1$ the image of $\cbar\smm\G_1$ under $v_1$. Then $X_1\supset X_0$ and there exists a bijection $e_1$ from $\G_1$ to the collection of edges of $(T_F, X_1)$ such that $e_1(\beta)\subset e(\g)$ if $\beta\in\G_1$ is isotopic to $\g\in\G$.

Each component $E_1$ of $\cbar\smm\G_1$ is cut by $F^{-1}(\G)$ into finitely many domains, there is exactly one of them, denoted by $\check E_1$, is not D-type. Define a map
$$
\tau_F:\  X_1\to X_0\quad\text{by}\quad\tau_F(v_1(E_1))=v(F(\check E_1)).
$$
If two points $v_1(E_1)$ and $v_1(E_2)$ in $X_1$ are connected by an edge in $(T_F, X_1)$, then $E_1$ and $E_2$ have a common boundary curve
$\beta\in\G_1$. So do $\check E_1$ and $\check E_2$. Thus $F(\check E_1)$ and $F(\check E_2)$ have a common boundary curve $F(\beta)\in\G$. So $\tau_F(v_1(E_1))$ and $\tau_F(v_2(E_2))$ are connected by an edge in $(T_F, X_0)$. Therefore, we can extend the map $\tau_F$ to a continuous map $\tau_F: T_F\to T_F$ such that $\tau_F:\, e_1(\beta)\to e(F(\beta))$. Moreover, we may equip a linear metric on $T_F$ such that $\tau_F$ is linear on each edge of $(T_F, X_1)$. The tree map
$$
\tau_F:\,(T_F, X_1)\to (T_F,X_0)
$$
will be called the {\bf Shishikura tree map} of $F$.

In the following, we will show that the transition matrix of the multicurve $\G$ can be expressed by the Shishikura tree map together with the degrees of $F$ on curves in $\G_1$.

By a {\bf weight} of a tree we mean a positive function defined on the collection of edges of the tree.

Let $T$ be a finite tree with vertex set $X\subset T$ and $\tau:\ T\to T$ be a tree map. Note that $\tau^{-1}(X)\supset X$ and $(T, \tau^{-1}(X))$ is a new tree with the vertex set $\tau^{-1}(X)$. Let $w$ be a weight on $(T, \tau^{-1}(X))$. Denote by $\{I_1, \cdots, I_n\}$ the edges of $(T, X)$. The transition matrix $M(\tau, w)=(b_{ij})$ of $\tau$ with respect to the weight $w$ is defined by
$$
b_{ij}=\sum_{J}\frac 1{w(J)},
$$
where the sum is taken over all the edges $J$ of $(T, \tau^{-1}(X))$ such that $J\subset I_{i}$ and $\tau(J)=I_{j}$. From Lemma \ref{eigenvalue}, we have

\begin{lemma}\label{tree-eigen} The leading eigenvalue satisfies $\lambda(M(\tau, w))<1$ if and only if $\tau$ is contracting with respect to the weight $w$, i.e., there exists a linear metric $\rho$ on $T$ such that for each edge $I$ of $(T,X)$,
$$
\sum_{J}\frac{|\tau(J)|}{w(J)}<|I|,
$$
where the sum is taken over all the edges $J$ of $(T, \tau^{-1}(X))$ with $J\subset I$ and $|\cdot|$ denotes the length of edges under the metric $\rho$.
\end{lemma}

Let $\tau_F:\,(T_F, X_1)\to (T_F,X_0)$ be the Shishikura tree map of $F$. Define the weight for edges $J=e_1(\delta)$ of $(T, X_1)$ by $w_F(J)=\deg(F|_\delta)$. Then the transition matrix $M(\tau_F, w_F)$ is just the transition matrix of the canonical multicurve $\G$.

Now we consider the dynamics of a tree map $\tau:\, T\to T$. Let $X_0\subset T$ be the vertex set. Denote $X=\bigcup_{n\ge 0}\tau^{-n}(X_0)$.

\begin{lemma}\label{interval}
Suppose that $T\smm\overline{X}\ne\emptyset$. Then for every component $J$ of $T\smm\overline{X}$, there exists an integer $N\ge 0$ such that $\tau^n(J)$ is an edge of $(T,X_0)$ for all $n\ge N$.
\end{lemma}

\begin{proof}
Let $\III$ denote the collection of edges of $(T, X_0)$ which contain points of $X$. Then there exists an integer $m>0$ such that each edge in $\III$ contains points of $\tau^{-m}(X_0)$. So there is a constant $0<\lambda<1$ such that for each edge $I\in\III$, the length of every interval of $I\smm \tau^{-m}(X_0)$ is less than $\lambda |I|$.

Let $J$ be a component of $T\smm\overline{X}$. Assume by contradiction that $\tau^n(J)$ is not an edge of $(T,X_0)$ for all $n\ge 0$. Then $J$ is contained in an edge $I^0\in\III$. Let $I^1$ be the component of $I^0\setminus\tau^{-m}(X_0)$ that contains $J$. Then $|I^1|<\lambda|I^0|$.

Now $\tau^m(I^1)$ is an edge of $(T, X_0)$, which is also contained in $\III$ by the assumption. So $I^1$ contains points of $\tau^{-2m}(X_0)$.
Let $I^2$ be the component of $I^1\setminus\tau^{-2m}(X_0)$ that contains $J$. Then $|I^2|<\lambda|I^1|$.

Inductively, we obtain an infinite sequence of intervals $\{I^k\}$ with $I^k\supset I^{k+1}\supset J$ and $|I^k|<\lambda |I^{k-1}|$ for $k\ge 1$.
Thus $|J|<\lambda^k|I^0|\to 0$ as $k\to\infty$. This is a contradiction.
\end{proof}

\begin{corollary}\label{2sides}
Each point $x\in\overline{X}\smm X$ is a double-sides accumulation point, i.e., let $I$ be the edge of $(T, X_0)$ that contains the point $x$, then both of the two components of $I\smm\{x\}$ contain a sequence of points in $X$ which converges to the point $x$.
\end{corollary}

Let $x\in T\smm X_0$ be a periodic point with period $p\ge 1$. Then either $x\in T\smm\overline{X}$ or $x\in\overline{X}\smm X$. In the former case,
let $I$ be the edge of $T$ that contains the point $x$, then $\tau^p(I)=I$ by the above lemma. So $|(\tau^p)'(x)|=1$ on $I$ since $\tau$ is a linear map. In the latter case, $|(\tau^p)'(x)|>1$ and hence $x$ is a repelling periodic point.

\vspace{2mm}
\noindent
{\bf Remark.} The general converse problem, that is, given a Shishikura tree map, under what conditions, one could construct a rational map to realize such a tree map, is not involved in this paper. The realization problem was discussed in \cite{Shi3}, and moreover, a lower bound on the degree of rational maps was also given there.

\section{Jordan curves as components of Julia sets}
In this section, we will describe the dynamics on the configuration of Julia components by means of the Shishikura tree map.

Let $f$ be a non-degenerate and generic sub-hyperbolic rational map. A Julia component of $f$ is called {\bf buried} if it is disjoint from the closure of any Fatou domain of $f$. Obviously, a Julia component $K$ is buried if and only if $f(K)$ is also buried. Denote

\vskip 0.24cm
$\sA=\{\text{A-type buried Julia components of $f$ which are Jordan curves}\}$,

$\sC_0=\{\text{Q-type components of $\JJJ_f$ or $\FFF_f$}\}$,

$\sC_n=\{\text{A-type or Q-type components $K$ of $\JJJ_f$ or $\FFF_f$ such that $f^n(K)$ is Q-type}\}$,

$\sC=\bigcup_{n\ge 0}\sC_n$.

\vskip 0.24cm
Then $\sA$ is disjoint from $\sC$ and $f(K)\in\sA$ for $K\in\sA$,  $\sC_{n}\supset\sC_{n-1}$ and $f(K)\in\sC_{n-1}$ for $K\in\sC_n$ and $n\ge 1$.

\begin{theorem}\label{tree}
Let $\Gamma$ be a canonical multicurve of $f$ and $\tau:\,(T, X_1)\to (T, X_0)$ be the Shishikura tree map associated with $\Gamma$. Denote $X=\bigcup_{n\ge 0}\tau^{-n}(X_0)$. Then there exists an order-preserving injection $\pi:\, \sC\cup\sA\to T$ such that $\pi(\sC)=X$, $\pi(\sA)=\overline{X}\smm X$ and the following diagram commutes.
\[
\xymatrix{
\sA\cup\sC\ar[d]_{\pi}\ar[r]^{f} & \sA\cup\sC\ar[d]_{\pi} \\
\overline{X}\ar[r]^{\tau} & \overline{X}}
\]
The injection $\pi$ is order-preserving means that for distinct components $K_0, K_1$ and $K_2$ in $\sA\cup\sC$, $K_0$ separates $K_1$ from $K_2$ if and only if $\pi(K_0)$ separates $\pi(K_1)$ from $\pi(K_2)$ in $T$. In particular, the injection $\pi$ is a one-to-one correspondence from the set of periodic Jordan curves as buried Julia components of $f$ to the set of repelling periodic points of $\tau$ in $T\smm X_0$.
\end{theorem}

\begin{proof}
By Lemma \ref{Y2C}, for each $K\in\sC_0$, there exists a unique component $E$ of $\cbar\smm\G$ such that $K\subset E$ if $K$ is a Julia component, or both $E$ and $K$ contain a common Q-type component $U$ of $\UUU$ if $K$ is a Fatou domain. Thus we have an injection
$$
\pi: \sC_0\to T, \quad \pi(K)=v(E)
$$
such that $\pi(\sC_0)=X_0$. It is order-preserving since $v$ is order-preserving.

Let $\Gamma_1$ be the collection of essential curves in $f^{-1}(\Gamma)$. For each $K\in\sC_1$, there exists a unique component $E_1$ of $\cbar\smm\G_1$ such that $K\subset E_1$ when $K$ is a component of $\JJJ_f$, or both $E_1$ and $K$ contain a common Q-type or A-type component $U$ of $\UUU$. Thus $\pi$ can be extended to
$$
\pi: \sC_1\to T, \quad \pi(K)=v_1(E_1)
$$
such that $\pi(\sC_1)=X_1$ and $\tau(\pi(K))=\pi(f(K))$ for $K\in \sC_1$. The injection $\pi$ is still order-preserving since $v_1$ is order-preserving.

Since $\tau$ is injective on each edge of $(T, X_1)$, for any $n\ge 2$, $\pi$ can be extended to an order-preserving injection $\pi: \sC_n\to T$ such that $\pi(\sC_n)=\tau^{-n}(X_0)$ and the following diagram commutes.
\[
\xymatrix{
\sC_n\ar[d]_{\pi}\ar[r]^{f} & \sC_{n-1}\ar[d]_{\pi}\ar[r] & \cdots\ar[r] & \sC_1\ar[d]_{\pi}\ar[r]^{f} & \sC_0\ar[d]^{\pi} \\
\tau^{-n}(X_0)\ar[r]^{\tau} & \tau^{-n+1}(X_0)\ \ar[r] & \cdots \ar[r] &
\tau^{-1}(X_0)\ar[r]^{\tau} & X_0 }
\]
In conclusion, we obtain an order-preserving injection $\pi: \sC\to T$ such that $\pi(\sC)=X$ and the following diagram commutes.
\[
\xymatrix{
\sC\ar[d]_{\pi}\ar[r]^{f} & \sC\ar[d]_{\pi} \\
{X}\ar[r]^{\tau} & {X}}
\]

For each $K\in\sA$, $K$ is Jordan curve as a buried Julia component. So there exists an infinite sequence of pairs of annular Fatou domains $(U_n, V_n)$ such that $K\subset A_{n+1}\subset A_n$ and $A_n$ is disjoint from all the elements in $\sC_n$, where $A_n$ is the unique annular component of $\cbar\smm(\overline{U_n}\cup\overline{V_n})$.  It concludes that both $\pi(U_n)$ and $\pi(V_n)$ converge to the same point $x\in\overline{X}\smm X$. Define $\pi(K)=x$. Then we extend $\pi$ from $\sC$ to $\sC\cup\sA$, which is still an order-preserving injection.

For each point $x\in\overline{X}\smm X$. Let $I_0$ be the edge of $(T, X_0)$ that contains the point $x$. By Corollary \ref{2sides}, $x$ is a double-sides accumulation point. Thus there exists an infinite sequence of intervals $\{I_n=[a_n, b_n]\}$ with $x\in I_n\Subset I_{n-1}$, such that both $a_n$ and $b_n$ are contained in $X$ and $|I_n|\to 0$ as $n\to\infty$.

Let $A_n=\pi^{-1}(a_n)$ and $B_n=\pi^{-1}(b_n)$ be the corresponding components of $\JJJ_f$ or $\FFF_f$. Let $C_n$ be the unique annular component of $\cbar\smm(\overline{A_n}\cup\overline{B_n})$. Then $C_{n+2}\Subset C_n$ and $C_{n+2}$ separates the two complementary components of $C_n$ for all $n\ge 0$.
Thus $K=\bigcap_{n\ge 0}C_n$ is a continuum which has exactly two complementary components. Thus $\partial K$ has at most two components. On the other hand, $\partial K\subset\JJJ_f$ and is disjoint from the grand orbit of all the complex periodic Julia components. Thus each component of $\partial K$ is contained in a wandering Julia component or an eventually simple periodic Julia component. In both cases each component of $\partial K$ must be a Jordan curve.

If $\partial K$ has two components, then there exists an A-type Fatou domain $U$ which separates one of them from another. Thus $\pi(U)\in I_n$ for all $n\ge 0$. This is a contradiction since $|I_n|\to 0$. Thus $\partial K$ is a Jordan curve and so is $K$.

Note that $K$ is disjoint from the closure of periodic Fatou domains. Applying the above argument for $\tau^n(x)$, we obtain another Jordan curve $K_n$ as a Julia component and $K_n$ is also disjoint from the closure of periodic Fatou domains. From $\pi\circ f=\tau\circ\pi$ on $X$, we obtain $K_n=f^n(K)$. This shows that $K$ is disjoint from the closure of any Fatou domain. Thus $K$ is a buried Julia component. This shows that $\pi(\sA\cup\sC)=\overline{X}$.
Now the proof is complete.
\end{proof}

For our purpose, we want to know whether $\sA$ is an infinite collection, or equivalently, whether $\overline{X}\smm X$ is an infinite set by Theorem \ref{tree}. The next lemma provides a necessary and sufficient combinatorial condition. Refer to \cite{CPT} for the next definition.

Let $\G$ be a multicurve. Denote by $\G_1$ the collection of curves in $F^{-1}(\G)$ isotopic to a curve in $\G$, and denote by $\G_n$ the collection of curves in $F^{-1}(\G_{n-1})$ isotopic to a curve in $\G$ for $n\ge 2$. For each $\g\in\G$, denote
$$
\kappa_n(\g)=\#\{\beta\in\G_n:\,\ \beta\text{ is isotopic to $\g$}\}.
$$
The multicurve $\G$ is a {\bf Cantor multicurve} if $\kappa_n(\g)\to\infty$ as $n\to\infty$ for all $\g\in\G$.

\begin{lemma}\label{cantor multicurve}
Let $f$ be a non-degenerate and generic sub-hyperbolic rational map. Let $\G$ be a canonical multicurve of $f$ and $\tau: (T, X_1)\to(T, X_0)$ be the Shishikura tree map of $f$. The following conditions are equivalent.

(a) $\G$ contains a Cantor multicurve.

(b) $\tau$ has infinitely many repelling periodic cycles.

(c) $\tau$ has wandering points.
\end{lemma}

\begin{proof}
(a) $\Rightarrow$ (b) and (c).
Let $\G_0\subset\G$ be a Cantor multicurve. Denote by $I_k$ ($1\le k\le n$) the edges of $(T, X_0)$ corresponding to $\G_0$. Set $T_0=\cup_{k=1}^n I_k$
and $T_1=T_0\cap\tau^{-1}(T_0)$. Then $T_1\subset T_0$ and $\tau(T_1)\subset T_0$. Moreover, $\tau$ maps each component of $T_1$ onto one edge.

Denote $\tau_0=\tau|_{T_1}$. The multicurve $\G_0$ is a Cantor multicurve implies that for any integer $M>0$, there is an integer $N>0$ such that $\tau_0^{-N}(T_0)\cap I_k$ contains at least $M$ components for each $1\le k\le n$.

Take $M=n+1$. Then for some $N>0$, $\tau_0^{-N}(T_0)\cap I_k$ contains at least $n+1$ components for each $1\le k\le n$. Thus at least two of them map to the same edge under $\tau_0^N$. Define a map $\sigma$ on the index set $\{1, \cdots, n\}$ by $\sigma(i)=j$ if $\tau_0^{-N}(T_0)\cap I_i$ has two components mapping to $I_j$. Here we need to point out that the definition of $\sigma$ is not uniquely determined. As the index set is finite, each index is eventually periodic. In particular, there is a periodic index. By relabelling the edges, we may assume that the index $1$ is a periodic index with period $p\ge 1$. This implies that $\tau_0^{-pN}(T_0)\cap I_1$ has at least $2^p$ components mapping to $I_1$ by $\tau^{pN}$.

Let $J_0$ and $J_1$ be two distinct components of $\tau_0^{-pN}(T_0)\cap I_1$ mapping onto $I_1$ by $\tau^{pN}$. Write $\omega=\tau^{pN}|_{J_0\cup J_1}$ for simplicity. It is classical that the linear map $\omega: J_0\cup J_1\to I_1$ contains infinitely many repelling periodic cycles and wandering points.
So does $\tau$.

(b) $\Rightarrow$ (a). Suppose that $\tau$ has infinitely many repelling periodic cycles. Then there is an edge $I$ of $T$ such that it contains two repelling periodic points $x_1$ and $x_2$ with periods $p_1, p_2\ge 1$ respectively. For $i=1,2$, let $J_i$ be the component of $\tau^{-p_i}(I)$ that contains the point $x_i$. Denote
$$
J_i^n=(\tau^{p_i}|_{J_i})^{-n}(I).
$$
Then $J_i^n\to x_i$ as $n\to\infty$. Thus there is an integer $k>0$ such that $J_1^k$ is disjoint from $J_2^k$. Thus $\tau^{-kp_ip_j}(I)$ has at least two components contained in $I$, one is contained in $J_1^k$, and another is contained in $J_2^k$.

Let $\g\in\G$ be the curve corresponding to the edge $I$. Then $f^{-kp_1p_2}(\g)$ has at least two curves isotopic to $\g$. Let $\G_0\subset\G$ be a sub-multicurve defined by $\beta\in\G_0$ if $f^{-n}(\g)$ has a component isotopic to $\beta$ for some $n>0$. It is well-defined since $\G$ is completely stable. It is easy to check that $\G_0$ is a Cantor multicurve.

(c) $\Rightarrow$ (a).
Suppose that $x\in T$ is a wandering point of $\tau$. Then there exists an edge $I$ of $T$ such that it contains infinitely many points of the forward orbit of $x$. Denote by
$$
0<n_1<n_2<\cdots<n_k<\cdots
$$
a sequence of integers such that $\tau^{n_k}(x)\in I$. Let $J_1^k$ be the component of $\tau^{n_1-n_k}(I)$ that contains $\tau^{n_1}(x)$. Then $J_1^k$ is disjoint from $\tau^{n_1-n_k}(X_0)$. Thus $|J_1^k|\to 0$ as $k\to\infty$ since $\tau^i(x)\in\overline{X}\smm X$ is a double-sides accumulation point by Corollary \ref{2sides}.

Let $J_2^k$ be the component of $\tau^{n_2-n_k}(I)$ that contains $\tau^{n_2}(x)$. Then $|J_2^k|\to 0$ as $k\to\infty$. Thus there is an integer $k>2$ such that $J_1^k$ is disjoint from $J_2^k$. This shows that $\tau^{-(n_1-n_k)(n_2-n_k)}(I)$ has at least two components contained in $I$, one is contained in $J_1^k$, and another is contained in $J_2^k$. The same argument as above shows that $\G$ contains a Cantor multicurve.
\end{proof}

\section{Self-grafting}
In this section, we will introduce the procedure of self-grafting on a tree map to construct a new tree map, and apply Theorem \ref{obstruction} to realize the new Shihikura tree map so that the corresponding new rational map has a new cycle of complex Julia components.

Before the discussion of the general procedure of self-grafting, we would like to give a simple example showing a tree map and its self-grafting.

\noindent {\bf Example.} Let $\tau$ be a tree map illustrated as in Figure 2, where
$$
X_0=\{a,b\},\ X_1=\{a,b,a_{-1},b_{-1}\}
$$
and
$$
\tau:\  a\mapsto a, a_{-1}\mapsto a, b\mapsto b, b_{-1}\mapsto b.
$$
Suppose that $\tau(x_0)=x_0$, where $x_0$ is a point on the edge $(b_{-1},a_{-1})$ with vertices $b_{-1}$ and $a_{-1}$ of $(T,X_1)$. Denote by $(x_0,b)$ the edge with vertices $x_0$ and $b$ of $(T,X_0)$. Set $B:=(x_0,b)$.

\begin{figure}[htbp]\centering
\includegraphics[width=12.5cm]{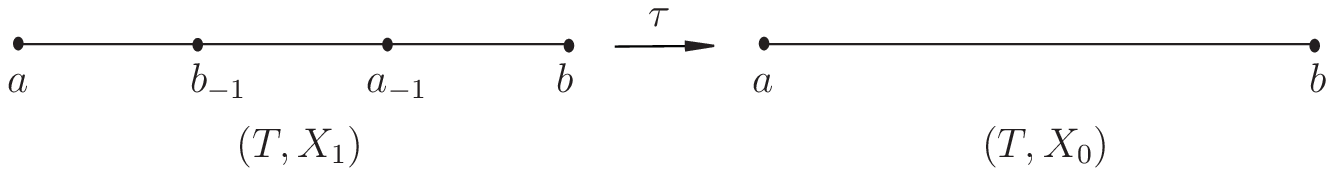}
\label{exampletree1}
\end{figure}

\begin{figure}[htbp]\centering
\includegraphics[width=12.5cm]{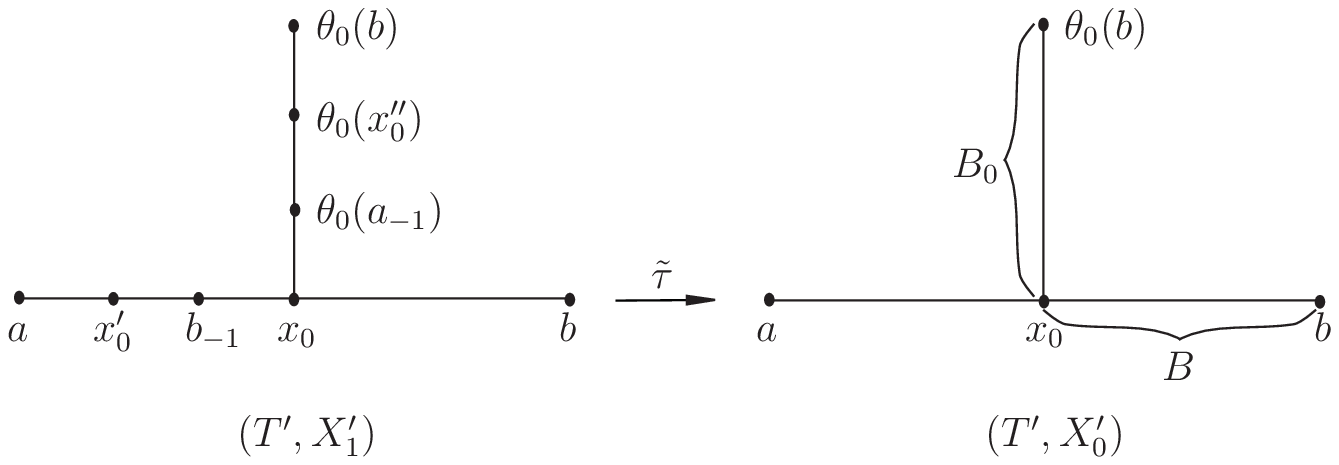}
\begin{center}{\sf Figure 2.}
\end{center}
\label{exampletree2}
\end{figure}

Define a new tree $T'\supset T$ as following (refer to Figure 2): $T'\smm T$ has exactly one component, denoted by $B_0$, satisfying that $B_0$ is homeomorphic to $B$ by a linear map $\theta_0:\, B\to B_0$, and $B_0$ is connected to $T$ at the point $x_0$. Denote by $X_0'$ the vertex set of $T'$ and $X_0'=\{a,b,x_0,\theta_0(b)\}$. Define
$$
\tau_0=
\begin{cases}
\tau:\, T\to T, \\
\mathrm{id}:\, B_0\to B_0, \\
\end{cases}
$$
and
$$
\kappa=
\begin{cases}
\mathrm{id}:\, T'\smm(B\cup B_0)\to T'\smm(B\cup B_0), \\
\theta_0:\, B\to B_0, \\
\theta_0^{-1}:\, B_0\to B.
\end{cases}
$$
Set $\tilde\tau=\tau_0\circ\kappa:\ T'\to T'$ and $X'_1=\tilde\tau^{-1}(X'_0)$. It is easy to verify that
$$
X'_1=\{a,x_0',b_{-1},x_0,b,\theta_0(a_{-1}),\theta_0(x_0''),\theta_0(b)\},
$$
where $x_0'$ and $x_0''$ are the preimages of $x_0$ under the map $\tau$ in the edges $(a,b_{-1})$ and $(a_{-1},b)$ of $(T,X_0')$ respectively, and
$$
\tilde\tau=
\begin{cases}
\tau:\ a\mapsto a,x_0'\mapsto x_0,b_{-1}\mapsto b,x_0\mapsto x_0, \\
\theta_0:\ b\mapsto \theta_0(b), \\
\tau\circ\theta_0^{-1}: \theta_0(a_{-1})\mapsto a,\theta_0(x_0'')\mapsto x_0,\theta_0(b)\mapsto b.
\end{cases}
$$
The new tree map $\tilde\tau:\, (T', X'_1)\to(T', X'_0)$ is called a self-grafting of the tree map $\tau$.

Now we introduce the general definition of a self-grafting. Let $\tau:\, (T, X_1)\to(T, X_0)$ be a tree map. Suppose that $O$ is a repelling periodic cycle in $T\smm X_0$ with period $p\ge 1$. Since $T$ is a tree (containing no loops), $T\smm O$ has a component whose boundary contains exactly one point in $O$. Denote this component by $B$ and this point by $x_0$. Obviously, $B$ is a component of $T\smm\{x_0\}$ and $B\cap O=\emptyset$.

Define a new tree $T'\supset T$ by the following (refer to Figure 3): there are exactly $p$ components $B_i$ of $T'\smm T$, $0\le i\le p-1$, each of them is homeomorphic to $B$ by a linear map $\theta_i:\, B\to B_i$, and $B_i$ is connected to $T$ at the point $x_i=\tau^i(x_0)$. The vertices $X'_0$ on $T'$ is assigned to be the original vertices on $T$ together with $O\cup\bigcup_{i=0}^{p-1}\theta_i(B\cap X_0)$.

Define a tree map $\tilde\tau: T'\to T'$ by $\tilde\tau=\tau_0\circ\kappa$, where
$$
\tau_0=
\begin{cases}
\tau:\, T\to T, \\
\theta_{i+1}\circ\theta_i^{-1}:\, B_i\to B_{i+1} \,\text{ for }0\le i<p-1, \\
\theta_0\circ\theta_{p-1}^{-1}:\, B_{p-1}\to B_0, \\
\end{cases}
$$
and
$$
\kappa=
\begin{cases}
\mathrm{id}:\, T'\smm(B\cup B_0)\to T'\smm(B\cup B_0), \\
\theta_0:\, B\to B_0, \\
\theta_0^{-1}:\, B_0\to B.
\end{cases}
$$
Set $X'_1=\tilde\tau^{-1}(X'_0)$. We call the new tree map $\tilde\tau:\, (T', X'_1)\to(T', X'_0)$ a {\bf self-grafting} of the tree map $\tau$ along the orbit $\{x_0,\cdots, x_{p-1}\}$.

\begin{figure}[htbp]\centering
\includegraphics[width=9cm]{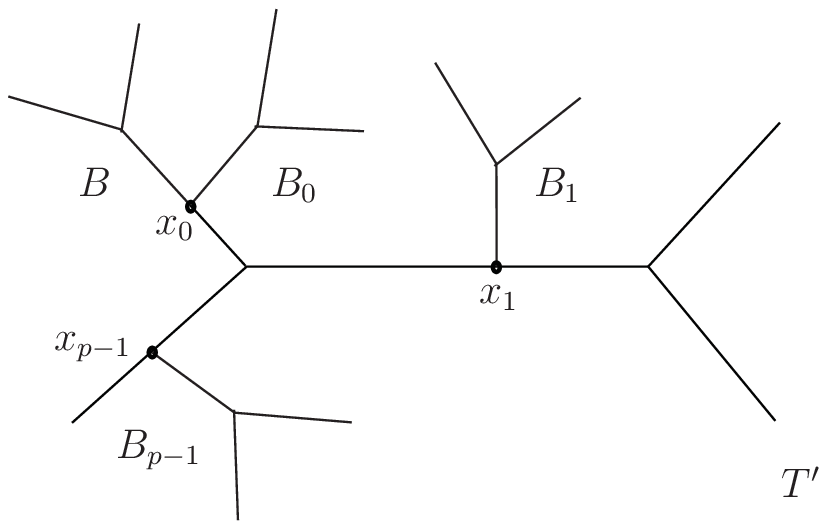}
\begin{center}{\sf Figure 3. A self-grafting along the orbit $\{x_0,\cdots, x_{p-1}\}$}
\end{center}
\label{self-grafting-1}
\end{figure}

By the definition, $\tilde\tau^i=\theta_i: B\to B_i$ for $0<i<p$ and $\tilde\tau^p=\theta_0: B\to B_0$. Thus $\tilde\tau^{p+1}=\tau$ on $B$.
So the original tree map $\tau$ can be expressed by $\tilde\tau$ as
$$
\tau=
\begin{cases}
\tilde\tau\, & \text{ on } T\smm B, \\
\tilde\tau^{p+1}\, & \text{ on } B.
\end{cases}
$$
In conclusion, any periodic point of $\tau$ with period $q\ge 1$ is also a periodic point of $\tilde\tau$ with period $q+kp$, where $k\ge 0$ is the number of times at which the cycle passes through $B$. Conversely, any cycle of $\tilde\tau$ meeting some $B_i$ must pass through $B$ since $\tilde\tau^{-i}(B_i)=B$ for $0<i<p$ and $\tilde\tau^{-p}(B_0)=B$.

For any weight $w$ on the tree $(T, X_1)$, the {\bf induced weight} $\tilde w$ on the tree $(T', X'_1)$ is defined as the following:
for each edge $J'$ of $(T', X'_1)$,
$$
\tilde w(J')=
\begin{cases}
1, & \text{ if }J'\subset B\cup\bigcup_{i=1}^{p-1}B_i, \\
w(\kappa(J')), & \text{ if }J'\subset B_0, \\
w(J'),  & \text{ if }J'\subset T\smm B.
\end{cases} \eqno{(1)}
$$
Here we need to point out that if $J'\subset B_0$, then $\kappa(J')$ need not to be an edge of $(T, X_1)$. However, it must be contained in an edge $J$ of $(T, X_1)$ since $\tau(\kappa(J'))=\tilde\tau(J')$ is disjoint from $X_0$. If $J'\subset T\smm B$, then $J'$ is also contained in an edge $J$ of $(T, X_1)$ since $\tilde\tau(J')=\tau(J')$. We define $w(J')=w(J)$ in both cases. To clarify the relationship between the edges of $(T',X_1')$ and edges of $(T,X_1)$, we want to go back to the example given at the beginning of this section and check the above statements for the example. In the example, if $J'\subset B_0$, then
$$
J'=
\begin{cases}
(x_0,\theta_0(a_{-1}))\stackrel{\kappa=\theta_0^{-1}}{\longrightarrow} (x_0,a_{-1})\subset (b_{-1},a_{-1}), \\
(\theta_0(a_{-1}),\theta_0(x_{0}''))\stackrel{\kappa=\theta_0^{-1}}{\longrightarrow} (a_{-1},x_0'')\subset (a_{-1},b), \\
(\theta_0(x_{0}''),\theta_0(b))\stackrel{\kappa=\theta_0^{-1}}{\longrightarrow} (x_0'',b)\subset (a_{-1},b). \\
\end{cases}
$$
Note that both $(b_{-1},a_{-1})$ and $(a_{-1},b)$ are edges of $(T,X_1)$. Now suppose $J'\subset T\smm B$. Then
$$
J'=
\begin{cases}
(a,x_0)\subset (a,b_{-1}), \\
(x_0',b_{-1})\subset (a,b_{-1}), \\
(b_{-1},x_0)\subset (b_{-1},a_{-1}). \\
\end{cases}
$$
Obviously, both $(a,b_{-1})$ and $(b_{-1},a_{-1})$ are edges of $(T,X_1)$.
\begin{lemma}\label{I-eigenvalue}
Let $w$ be a weight on the tree $(T, X_1)$ and let $\tilde w$ be the induced weight on $(T', X'_1)$. Then $\lambda(M(\tilde\tau, \tilde w))<1$ if $\lambda(M(\tau, w))<1$.
\end{lemma}

\begin{proof}
At first, We consider the tree map $\tau''=\tau: (T, X''_1)\to (T, X''_0)$, where $X''_0=X_0\cup O$ and $X''_1=\tau^{-1}(X''_0)$.

Denote by $\{I_1, \cdots, I_n\}$ the edges of $(T, X_0)$. Denote by $\{I''_1, \cdots, I''_m\}$ the edges of $(T, X''_0)$. Let $\Theta=\{1,\cdots, m\}$ be the index set. It is divided into $\Theta=\Theta_1\sqcup\cdots\sqcup\Theta_n$ such that $I''_i\subset I_k$ if $i\in\Theta_k$.

The weight $w$ on $(T, X_1)$ induces a weight on $(T, X''_1)$, denoted also by $w$, such that $w(J'')=w(J)$ is $J''\subset J$. Let $M(\tau, w)=(a_{kl})$ and $M(\tau'', w)=(b_{ij})$ be the transition matrices. Then for each pair $(k,l)$ and any $j\in\Theta_l$,
$$
\sum_{i\in\Theta_k} b_{ij}=a_{kl} \eqno{(2)}
$$
since for all $j\in\Theta_l$, $\tau^{-1}(I''_j)$ have the same number of components in $I_k$.

Let $\lambda$ be the leading eigenvalue of $M(\tau'', w)$. Then there is a non-zero vector $v=(v_i)\ge 0$ such that $M(\tau'', w)v=\lambda v$. Let $u=(u_k)$ be a vector defined by $u_k=\sum_{i\in\Theta_k}v_i$. Then $u$ is also a non-zero vector with $u\ge 0$. Now the equation (2) implies that $M(\tau, w)u=\lambda u$.
This shows that $\lambda$ is also an eigenvalue of $M(\tau, w)$. Thus $\lambda\le\lambda(M(\tau, w))<1$.

Now let us compare the tree maps $\tilde\tau$ with $\tau''$. Each edge of $(T, X''_0)$ is also an edge of $(T', X'_0)$. By Lemma \ref{tree-eigen}, there exist a linear metric $\rho$ on $(T, X''_0)$ and a constant $0<\lambda_1<1$ such that for each edge
$I$ of $(T, X''_0)$,
$$
\sum_J\frac{|\tau(J)|}{w(J)}<\lambda_1^p|I|, \eqno{(3)}
$$
where the sum is taken over all the edges $J$ of $(T, X''_1)$ in $I$ and $|\cdot|$ denotes the length with respect to the metric $\rho$.

Define a linear metric $\rho_1$ on $T'$ such that for each edge $I$ of $(T', X'_0)$, $|I|_1=|I|$ if $I$ is contained in an edge of $(T,X_0'')$, and $|I|_1=\lambda_1^i|\theta_i^{-1}(I)|$ if $I\subset B_i$ for $0<i<p$, and $|I|_1=\lambda_1^p|\theta_0^{-1}(I)|$ if $I\subset B_0$, where $|\cdot|_1$ denotes the length under the metric $\rho_1$.

For each edge $I$ of $(T', X'_0)$ in $T\smm B$, $\tilde\tau(I)=\tau(I)$ is contained in an edge of $(T,X_0'')$. From (1) and (3), we have
$$
\sum_J\frac{|\tau(J)|_1}{\tilde w(J)}<|I|_1, \eqno{(4)}
$$
where the sum is taken over all the edges $J$ of $(T', X'_1)$ in $I$.

For each edge $I$ of $(T', X'_0)$ in $B$, $J=I$ is also an edge of $(T', X'_1)$ with $\tilde w(J)=1$ and $\tilde\tau(J)=\theta_1(J)$ is contained in $B_1$. Thus
$$
\frac{|\tilde\tau(J)|_1}{\tilde w(J)}=\lambda_1|\theta_1^{-1}(\tilde\tau(J))|=\lambda_1|J|=\lambda_1|I|_1<|I|_1.
$$

For each edge $I_i$ of $(T', X'_0)$ in $B_i$ with $1\le i<p$, $J_i=I_i$ is also an edge of $(T', X'_1)$ with $\tilde w(J_i)=1$ and $\tilde\tau(J_i)=\theta_{i+1}\circ\theta_i^{-1}(J_i)$ is contained in $B_{i+1}$ (set $B_p=B_0$ and $\theta_p=\theta_0$). Thus
$$
\frac{|\tilde\tau(J_i)|_1}{\tilde w(J_i)}=\lambda_1^{i+1}|\theta_{i+1}^{-1}\circ\tilde\tau(J_i)|=\lambda_1^{i+1}|\theta_i^{-1}(J_i)|=\lambda_1|I_i|_1<|I_i|_1.
$$
So the inequality (4) holds for edges in $B$ or $B_i$ for $1\le i<p$.

Let $I_0$ be an edge of $(T', X'_0)$ in $B_0$. Let $I$ be the edge of $B$ with $\theta_0(I)=I_0$. Then $\tau(J)=\tilde\tau(\theta_0(J))$ for each edge $J$ of $(T, X_1'')$ in $I$. From (3), we have
$$
|I_0|_1=\lambda_1^p |I|>\sum_J\frac{|\tau(J)|}{w(J)}=\sum_J\frac{|\tilde\tau(\theta_0(J))|_1}{\tilde w(\theta_0(J))},
$$
where the sum is taken over all the edges $J$ of $(T, X_1'')$ in $I$. So the inequality (4) holds for every edge in $B_0$. Now the lemma  follows from Lemma \ref{tree-eigen}.
\end{proof}

The key step in the proof of Theorem \ref{main} is the next theorem.

\begin{theorem}\label{self-grafting}
Let $f$ be a non-degenerate and generic sub-hyperbolic rational map. Suppose that $f$ has a periodic Jordan curve disjoint from $\PPP_f$ as a buried Julia component with period $p$. Denote this Julia component by $C_0$ and set $C_i=f^i(C_0)$ for $0<i<p$. Let
$$\tau_f:\, (T, X_1)\to (T, X_0)$$
be the Shishikura tree map of $f$ and $\pi_f$ be the projection obtained by Theorem \ref{tree}. Then there exists a non-degenerate and generic sub-hyperbolic rational map $g$ with $\deg g=\deg f$ such that the Shishikura tree map $\tau_g$ of $g$ is the self-grafting of $\tau_f$. More precisely, the procedure of self-grafting is operated along the repelling periodic cycle
$$
\{\pi_f(C_0),\pi_f(C_1),\cdots, \pi_f(C_{p-1})\}\subset T\smm X_0.
$$
Consequently,

(a) $N(g)=N(f)+1$, and

(b) if $\tau_f$ has infinitely many repelling periodic cycles, so does $\tau_g$.

Moreover, $g$ can be chosen to be hyperbolic when $f$ is hyperbolic.
\end{theorem}

\begin{proof}
Note that at least one component of $\cbar\smm\bigcup_{i=0}^{p-1} C_i$ is a Jordan domain, by relabelling the index, we assume that this Jordan domain is bounded by $C_0$. Then $C_1, \cdots, C_{p-1}$ are contained in the same complementary component of $C_0$.

Let $(\UUU,\LLL)$ be a canonical decomposition of $f$. Since $\bigcup_{n>0}f^{-n}(\UUU)=\FFF_f$, as $n$ is large enough, each $C_i$ is contained in an A-type component $L_i$ of $\cbar\smm f^{-n}(\UUU)$ for $0\le i<p$. We may assume $n=0$ for the simplicity. Drop all the D-type components of $\UUU$, the remaining is still a canonical decomposition of $f$. Thus we may assume that each component of $\UUU$ is not D-type. Let $L_i$ be the component of $\LLL$ that contains $C_i$. Then $L_i$ is a closed annulus disjoint from $\PPP_f$.

Pick a Jordan domain $\Delta_0\subset L_0$ such that $\overline{\Delta_0}$ is disjoint from $\partial L_0$. Then there is a component $\Delta'$ of $f^{-p}(\Delta_0)$ such that $\Delta'\subset L_0$ and $\Delta_i=f^i(\Delta')\subset L_i$ for $1\le i<p$. Refer to Figure 4.

\begin{figure}[htbp]
\begin{center}
\includegraphics[width=13.5cm]{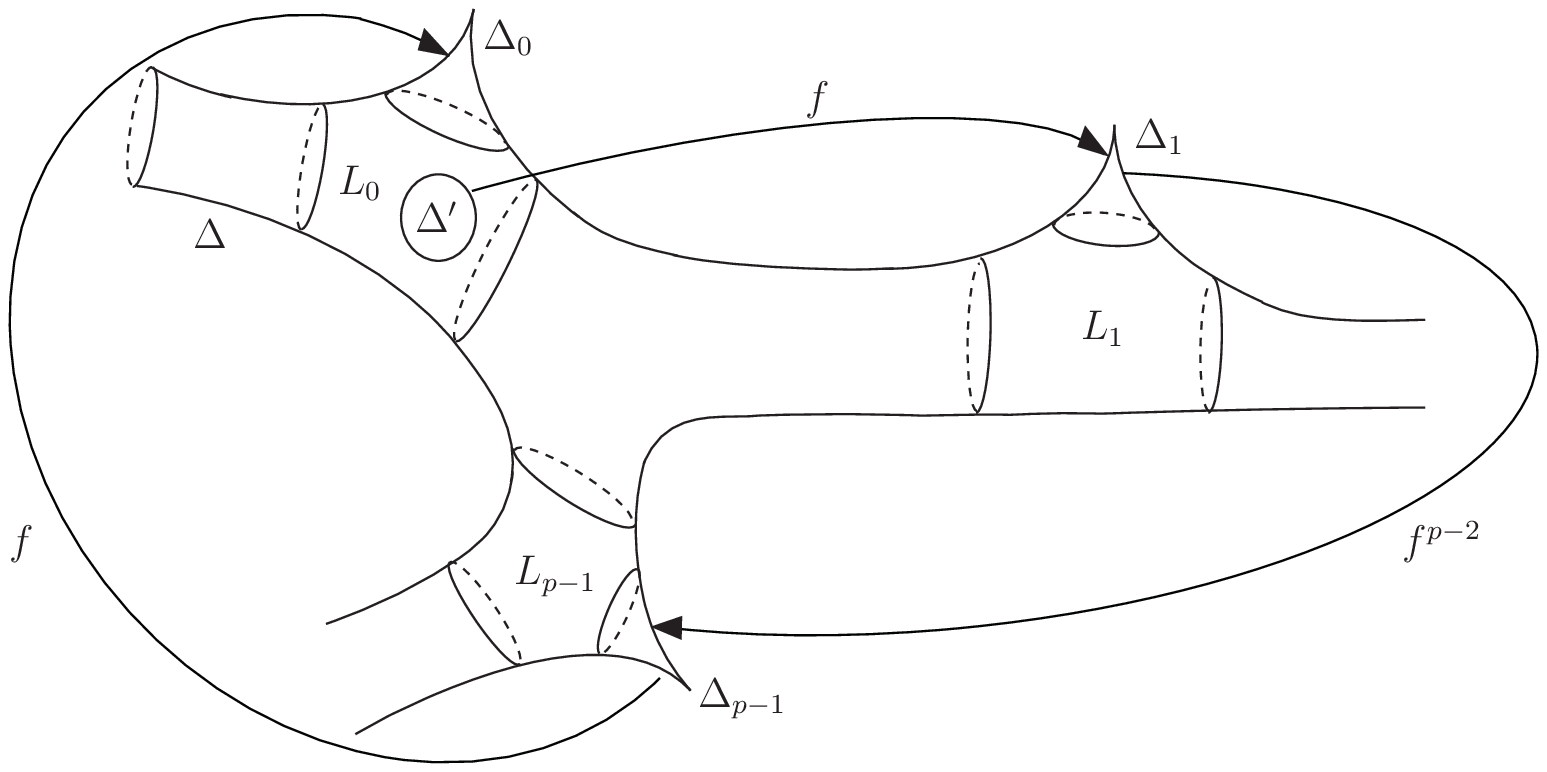}
\end{center}
\begin{center}{\sf Figure 4. The construction of $G$}
\end{center}
\end{figure}

There exists a homeomorphism $\phi_0$ of $\cbar$ such that $\phi_0=\mathrm{id}$ in $\cbar\smm L_0$ and $\phi_0=(f^p|_{\Delta'})^{-1}: \Delta_0\to\Delta'$. Set $F=f\circ\phi_0$. Then $F^p=\mathrm{id}$ in $\Delta_0$ and $F=f$ in $\cbar\smm L_0$.

Denote by $\Omega$ and $\Delta$ the two components of $\cbar\smm L_0$ such that $\Delta_i\subset\Omega$ for $1\le i<p$. Define a homeomorphism $\phi_1$ of $\cbar$ such that
$$
\begin{cases}
\phi_1=\text{\rm id}:\,\Omega\to\Omega, \\
\phi_1: \Delta\to\Delta_0\text{ is conformal}, \\
\phi_1=(\phi_1|_{\Delta})^{-1}: \Delta_0\to\Delta.
\end{cases}
$$

Set $G=F\circ\phi_1$. Then $G=f$ in $\Omega$. Since $F^p=\mathrm{id}$ in $\Delta_0$, we obtain $G^p=\phi_1$ in $\Delta$. Thus $G^{p+1}=f$ in $\Delta$ (see the next diagram).
\[
\xymatrix{
\Delta\ar[d]_{\phi_1}\ar[r]^{G} & \Delta_1\ar[r]^{G=f} & \cdots\ar[r]^{G=f} & \Delta_{0} &  \\
\Delta_0\ar[r]^{\phi_0} & \Delta' \ar[u]_{f} \ &  & \Delta\ar[r]^{\phi_0=\mathrm{id}}\ar[u]^{\phi_1} & \Delta\ar[u]_f }
\]

Conversely, $f$ can be expressed by $G$ as the following:
$$
f=
\begin{cases}
G\, & \text{ on }\Omega, \\
G^{p+1}\, & \text{ on }\Delta.
\end{cases}\eqno{(5)}
$$

By the definition, $G$ and $f$ have the same critical values. From $\PPP_f\subset\Omega\cup\Delta$, we obtain
$$
\PPP_G=\PPP_f\cup\bigcup_{i=1}^{p}P_i\subset\Omega\cup\Delta\cup\Delta_0,
$$
where $P=\Delta\cap\PPP_f$ and $P_i=G^i(P)\subset\Delta_i$ for $0<i\le p$ (set $\Delta_p=\Delta_0$). Moreover,
$$
G^{-i}(P_i)\cap\PPP_G=P. \eqno{(6)}
$$

Note that $G$ is holomorphic in $\Omega\cup\Delta\cup\Delta_0$. From (5), each periodic point of $f$ in $\PPP'_f$ with period $p'\ge 1$ is also a periodic point of $G$ with period $q=p'+kp$, where $k\ge 0$ is the number of times at which the cycle passes through $\Delta$.

From (6), any cycle of $G$ intersecting some $\Delta_i$ in $\PPP'_G$ must pass through $\Delta$. So for each cycle of $G$ in $\PPP'_G$ with period $q$, either it is a cycle of $f$ in $\PPP'_f\cap(\Omega\smm \cup_{i=1}^{p-1}\Delta_i)$, or it must pass through $\Delta$ and there exists a cycle of $f$ in $\PPP'_f$ with period $p'$ such that $q=p'+kp$, where $k$ is the number of times at which the cycle passes through $\Delta$. Therefore, each cycle of $G$ in $\PPP_G'$ is attracting or super-attracting. So $G$ is a semi-rational map.

Now we will construct a canonical decomposition for $G$. Recall that $(\UUU,\LLL)$ is a canonical decomposition of $f$. For the purpose of building a suitable canonical decomposition for $G$, we first modify  $(\UUU,\LLL)$ and give another canonical decomposition $(\UUU', \LLL')$ of $f$. Since $\UUU\Subset f^{-1}(\UUU)$, for each component $U$ of $\UUU$, there exists a tame domain $U'\Subset U$ such that

(i) $\partial U'$ is disjoint from $\PPP_f$,

(ii) each component of $U\smm\overline{U'}$ is an annulus disjoint from $\PPP_f$, and

(iii) $\UUU\Subset f^{-1}(\UUU')$, where $\UUU'=\bigcup_{U\subset\UUU}U'$.

Set $\LLL'=\cbar\smm\UUU'$. Then $(\UUU', \LLL')$ is also a canonical decomposition of $f$. For each component $U$ of $\UUU$ in $\Delta$, choose
$$
U'\Subset U^1\Subset\cdots\Subset U^{p-1}\Subset U^p\Subset U.
$$
Denote $U_i=G^i(U^i)$ for $1\le i\le p$. Then $U_i\Subset\Delta_i$. Set
$$
\UUU_G=\UUU'\cup(\bigcup_{U\subset\Delta}\bigcup_{i=1}^{p} U_i).
$$
Then $\UUU_G\Subset G^{-1}(\UUU_G)$. Denote $\LLL_G=\cbar\smm\UUU_G$. Then each Q-type component $L$ of $\LLL'$ is also a Q-type component of $\LLL_G$.
Moreover, if $L\subset\Delta$, then $G^i(L)$ is a Q-type component of $\LLL_G$ for $1\le i\le p$. There is an extra cycle of Q-type components of $\LLL_G$ consisting of $\{L'_0, \cdots, L'_{p-1}\}$ with $L'_i\supset L_i\smm\Delta_i$.

It is easy to check that $(\UUU_G, \LLL_G)$ is a canonical decomposition of $G$. Denote by
$$\tau_G: (T', X'_1)\to(T', X'_0)$$
the Shishikura tree map of $G$. Denote

$\bullet$ $x_i:=\pi_f(C_i)\in T$: the point corresponding to the Julia component $C_i$ for $0\le i<p$,

$\bullet$ $y_i\in T'$: the point corresponding to the Q-type component $L'_i$ of $\LLL_G$,

$\bullet$ $B$: the component of $T\smm\{x_0\}$ such that vertices in $B$ correspond to the Q-type components of $(\UUU, \LLL)$ contained in $\Delta$.

The reader may refer to Figures 3 and 4. Roughly speaking, in Figures 3 and 4, $\Delta$ corresponds to $B$, $\Delta_0$ corresponds to $B_0$ and $\Omega$ corresponds to $T\smm B$.

The above relations between Q-type components of $(\UUU', \LLL')$ and Q-type components of $(\UUU_G, \LLL_G)$ induce a linear injection $\iota:\, (T, X_0)\to (T', X'_0)$ such that $\iota(x_i)=y_i$ and a linear bijection from $B$ to some component of $T'\smm\{y_i\}$ for $0\le i<p$. Identify the tree $T$ with its image under the injection $\iota$. Then $T\subset T'$. It is easy to check that $\tau_G$ is a self-grafting of $\tau_f$.

The weight for the Shishikura tree map of $G$ exactly equals to the induced weight for the Shishikura tree map of $f$. Let $\G_G$ be a canonical multicurve of $G$. By Lemma \ref{I-eigenvalue}, $\lambda(\G_G)<1$.

The cycle of Q-type components of $\LLL_G$ consisting of $\{L'_0, \cdots, L'_{p-1}\}$ contains essentially no multicurve of $G$ since each $L'_i$ is disjoint from $\PPP_G$ and has only three complementary components.

Let $\G$ be a multicurve contained essentially in a periodic Q-type component $L$ of $\LLL_G$, where $L\ne L_i$ for all $0\le i\le p-1$. Let $q\ge 1$ be its period. The orbit of $L$ either is disjoint from $\Delta$ or passes $k$ times through $\Delta$. In the former case, $L$ is also a cycle of $\LLL$ with the same period. Thus $\lambda(\G,G^q)=\lambda(\G,f^q)<1$.

In the latter case, it contains a cycle of Q-type components of $\LLL$ with period $q_1\ge 1$ and $q=q_1+kp$. Thus when $L$ is a component of $\LLL$, $\lambda(\G, G^q)=\lambda(\G, f^{q_1})<1$. When $L$ is contained in $\Delta_i$ for $1\le i\le p$, $\G'=\{G^{-i}(\g), \g\in\G\}$ is a multicurve essentially contained in $\Delta$ and $\lambda(\G, G^q)=\lambda(\G', f^{q_1})<1$. From Theorem \ref{obstruction}, $G$ has no Thurston obstruction. Thus $G$ is c-equivalent to a rational map $g$.

It is obvious that (a) $N(g)=N(f)+1$ and (b) if $\tau_f$ has infinitely many repelling periodic cycles, so does $\tau_g$. Moreover, $g$ is hyperbolic when $f$ is hyperbolic.
\end{proof}

\section{Proof of Theorem \ref{main}}
In this section, based on Theorem \ref{self-grafting}, we will prove Theorem \ref{main}.

At first, we want to construct a hyperbolic rational map $f$ with $\deg f=3$ such that $N(f)=1$ and it has infinitely many periodic Jordan curves as buried components of $\JJJ_f$. Refer to Figure 5 for the construction.

\begin{figure}[htbp]\centering
\includegraphics[width=14cm]{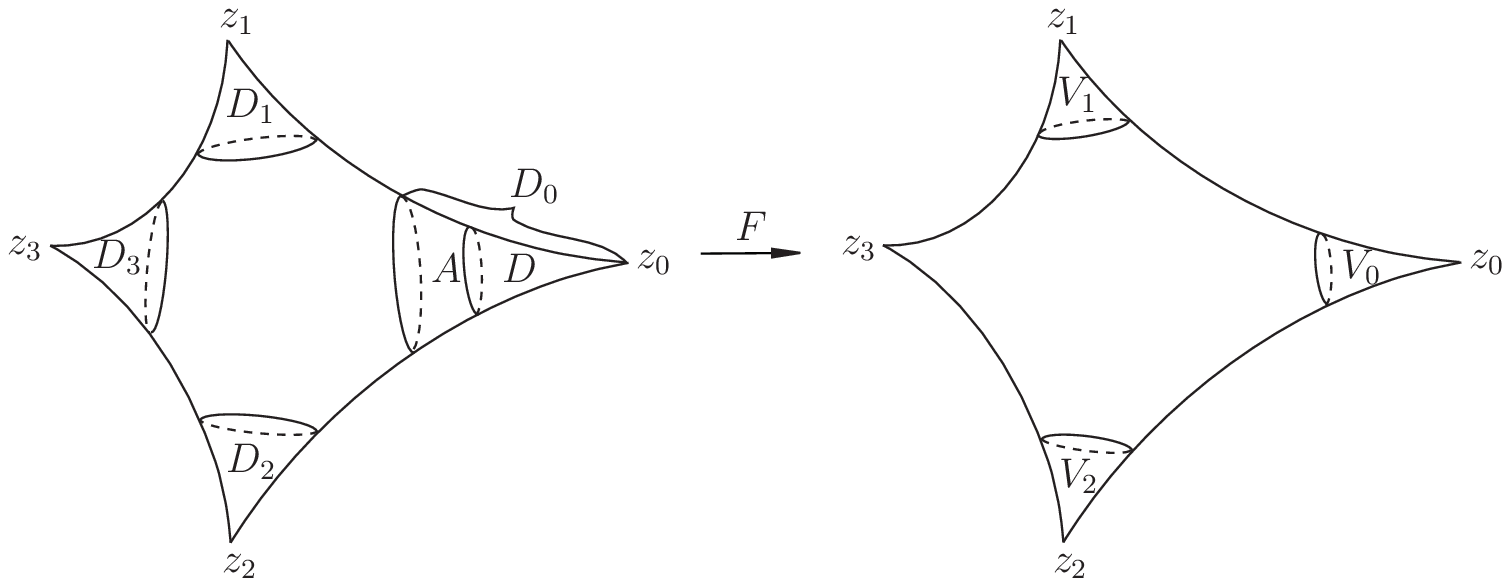}
\begin{center}{\sf Figure 5. The construction of $F$}
\end{center}
\end{figure}

We begin with the quadratic rational map
$$
h(z)=\frac{1}{(z-1)^2}.
$$
It has two critical points $z_1=1$ and $z_2=\infty$. Both of them are contained in the cycle
$$
z_1\mapsto z_2\mapsto z_0=0\mapsto z_1.
$$
So $\PPP_h=\{0,1,\infty\}$ and $\JJJ_h$ is connected.

Pick a B\"{o}ttcher disk $z_0\in V_0\Subset\FFF_h$. Then there are B\"{o}ttcher disks $V_1\ni z_1$ and $V_2\ni z_2$ such that $\VVV\Subset h^{-1}(\VVV)$, where $\VVV=\cup_{i=0}^2 V_i$.

Denote $z_3=2$. Then $h(z_3)=z_1$. The set $h^{-1}(\VVV)$ has $4$ components, denote them by $D_i$ such that $z_i\in D_i$.

Pick a Jordan curve $\al\subset D_0$ such that it separates the point $z_0$ from $\partial D_0$. Then $D_0\smm\al$ has two components: an annulus $A$ and a Jordan domain $D$.

Define a branched covering $F$ by the following:

(1) $F=h$ on $\cbar\smm D_0$,

(2) $F:\, A\to V_1$ is a branched covering with degree $2$, and

(3) $F:\, D\to \cbar\smm\overline{V}_1$ is a homeomorphism such that $F(z_0)=z_3$ and $F$ is holomorphic in a neighborhood of $z_0$.

Now $\{z_0, z_1, z_2, z_3\}$ is a super-attracting cycle of $F$. We may require that the two critical values of $F:\, A\to V_1$ are contained in a B\"{o}ttcher disk $U_1$ at $z_1$. Then $F$ is a semi-rational map with $\PPP'_F=\{z_0, z_1, z_2, z_3\}$.

There exist B\"{o}ttcher disks $U_i\ni z_i$ ($i=0,2,3$) such that $\UUU\Subset h^{-1}(\UUU)$, where $\UUU=\cup_{i=0}^3 U_i$. Since the two critical values of $F:\, A\to V_1$ are contained in $U_1$, we have $\PPP_F\subset\UUU$. Set $\LLL=\cbar\smm\UUU$. It is easy to check that $(\UUU, \LLL)$ is a canonical decomposition of $F$. The set $F^{-1}(\LLL)$ has two components, one is Q-type and the other is A-type.

Denote $\g_i=\partial U_i$ ($i=0,1,2,3$). Then $\G_F=\{\g_i\}_{i=0}^3$ is a canonical multicurve. Its transition matrix is:
$$
M=\left(\begin{array}{cccc}
0 & 2 & 0 & 1 \\
0 & 0 & \frac{1}{2} & 0 \\
\frac{1}{2} & 0 & 0& 0 \\
0 & 1 & 0 & 0
\end{array}\right).
$$

By a direct computation, we have
$$
M^3v=\left(\begin{array}{cccc}
\frac{1}{2} & 0 & \frac{1}{2} & 0 \\[5pt]
0 & \frac{1}{2} & 0 & \frac{1}{4} \\[5pt]
0 & \frac{1}{2} & \frac{1}{2} & 0 \\[5pt]
\frac{1}{4} & 0 & 0 & 0
\end{array}\right)
\left(\begin{array}{c}
v_1 \\[5pt]
v_2 \\[5pt]
v_3 \\[5pt]
v_4
\end{array}\right)
=
\left(\begin{array}{c}
\frac{1}{2}v_1+\frac{1}{2}v_3 \\[5pt]
\frac{1}{2}v_2+\frac{1}{4}v_4 \\[5pt]
\frac{1}{2}v_3+\frac{1}{2}v_2 \\[5pt]
\frac{1}{4}v_1
\end{array}\right).
$$
Choose the positive vector $v$ such that $v_4/2<v_2<v_3<v_1<4v_4$. Then $Mv<v$. So $\lambda(\G_F)=\lambda(M)<1$.

Let $\G$ be a muticurve of $F$ contained essentially in $\LLL$. Then $\G$ contains exactly one curve $\g$. Note that $F$ is a covering map from the Q-type component of $\LLL$ onto $\LLL$ with degree 2. If $\g$ separates $z_0$ from $z_2$, then $F^{-1}(\g)$ has only one component $\delta$ in the Q-type component of $F^{-1}(\LLL)$ and $\deg(F|_{\delta})=2$. So $\lambda(\G)<1$. If $\g$ does not separate $z_0$ from $z_2$, then each component of $F^{-1}(\g)$ in the Q-type component of $F^{-1}(\LLL)$ does not separate $z_1$ from $z_2$, On the other hand, $\g$ separates $z_1$ from $z_2$. Thus $\lambda(\G)=0$. Therefore $F$ is c-equivalent to a rational map $f$ by Theorem \ref{obstruction}.

One may also apply \cite[Theorem 2.1]{BCT} to show $\lambda(\G)<1$.

The Shishikura tree map $\tau:(T, X_1)\to (T, X_0)$ of $f$ is shown in Figure 6, where
$$
X_0=\{a_0,a_1,a_2,a_3, b\},\quad X_1=\{a_0, a_1, a_2, a_3, b, b_{-1}, a'_0\},
$$
and the tree map is uniquely determined by its definition on vertices:
$$
\tau: a_0\mapsto a_3\mapsto a_1\mapsto a_2\mapsto a_0,  a'_0\mapsto a_1, b_{-1}\mapsto b, b\mapsto b.
$$

\begin{figure}[htbp]\centering
\includegraphics[width=13cm]{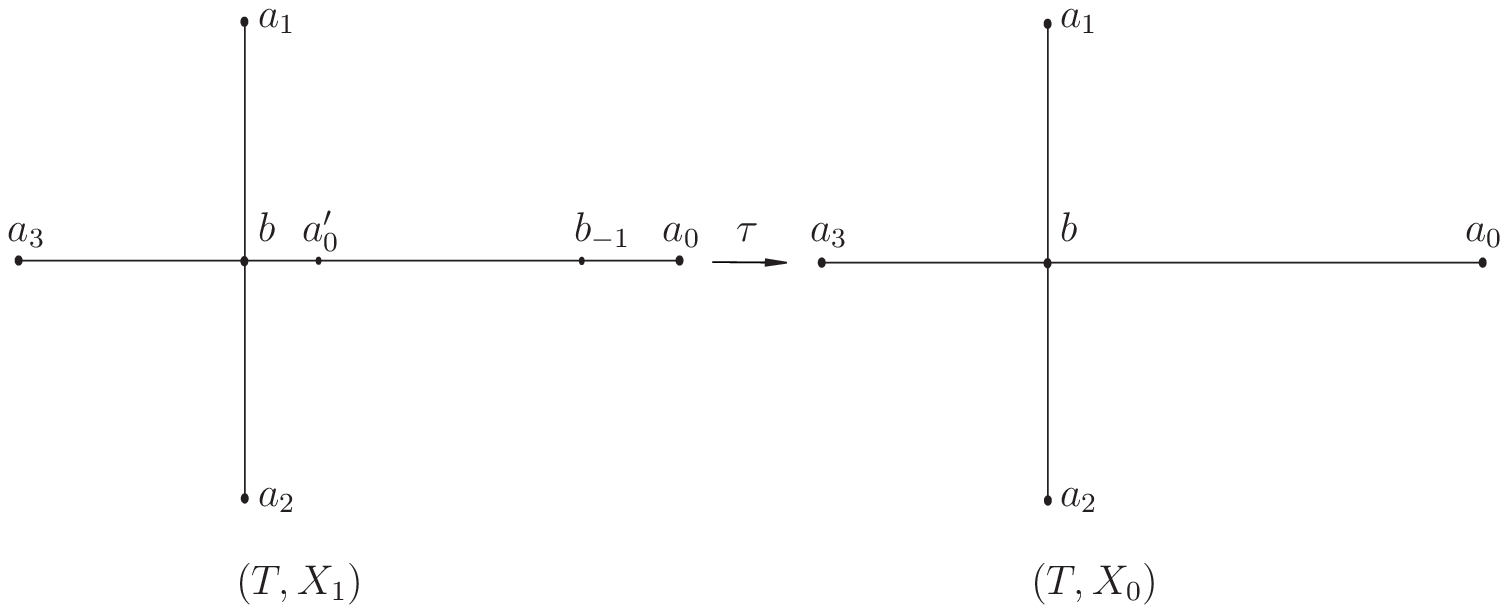}
\begin{center}{\sf Figure 6. The Shishikura tree map of $f$}
\end{center}
\end{figure}

One may refer to \cite{G} for a formula of the rational map $f$ and also the Julia set of $f$ (see Figure 11 in \cite{G} or the figure below).

\begin{figure}[htbp]\centering
\includegraphics[width=7.5cm]{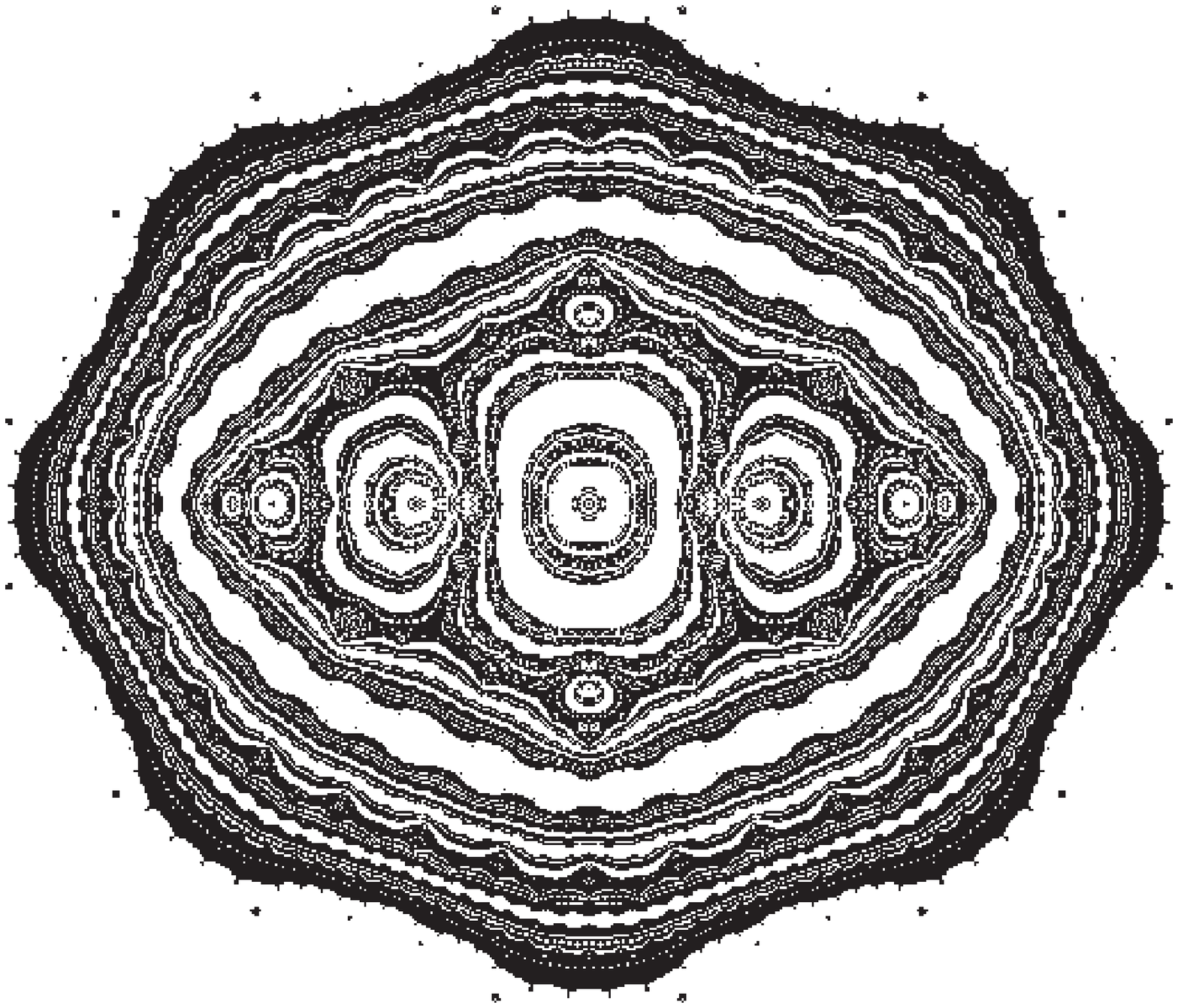}
\begin{center}{\sf Figure 7. A Persian carpet Julia set (in black)}
\end{center}
\end{figure}

 By the construction of $\Gamma_F$, it is not difficult to check that $\G_F$ is a Cantor multicurve. Thus by Lemma \ref{cantor multicurve}, the Shishikura tree map $\tau$ has infinitely many repelling periodic cycles (a repelling periodic cycle $\{x_0, x_1,\cdots, x_{20}\}$ is illustrated in Figure 8). Therefore, by Theorem \ref{tree}, the rational map $f$ has infinitely many periodic Jordan curves as buried Julia components.

\begin{figure}[htbp]\centering
\includegraphics[width=10cm]{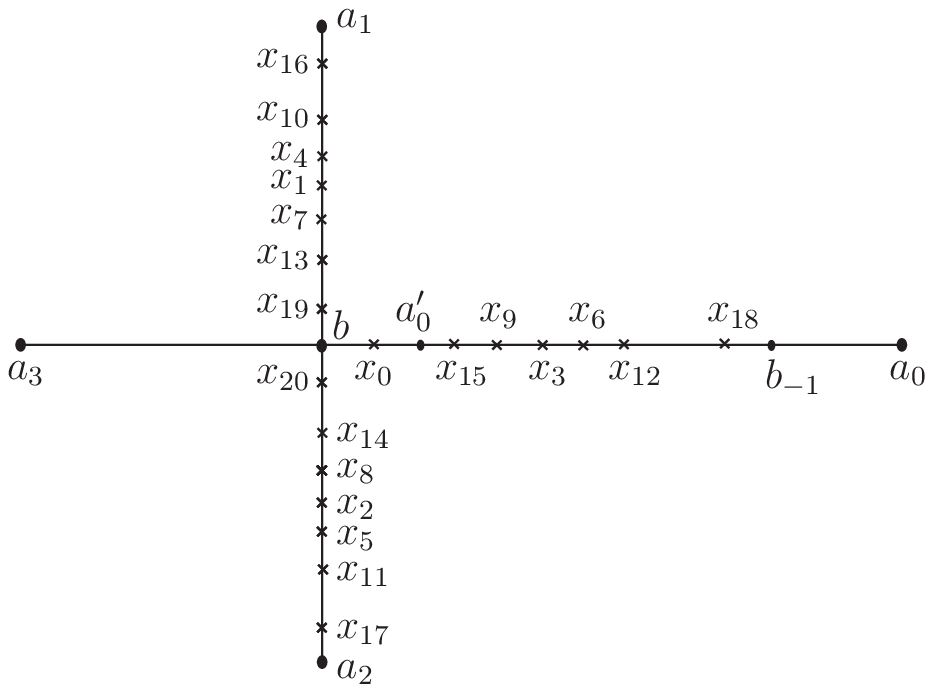}
\begin{center}{\sf Figure 8. A repelling periodic cycle of $\tau$}
\end{center}
\end{figure}

\vspace{2mm}
\noindent{\bf Remark.} We want to point out that the construction of the branched covering $F$ is different from the disc-annulus surgery in \cite{PT2}. In fact, the disc-annulus surgery is a procedure of quasi-conformal surgery. The branched covering constructed by such surgery has the property that for each point in $\cbar$, the forward orbit under the iterations of the map passes through the set consisting of non-holomorphic points at most a bounded number of times, and then by means of Shishikura's principle (see Lemma 1 in \cite{Shi1}), one could obtain a rational map. But in our situation, the branched covering $F$ does not necessarily satisfy that condition, and we apply Theorem \ref{obstruction} relating to Thurston's theorem to get a rational map realizing $F$.

\vskip 0.24cm
{\noindent\it Proof of Theorem \ref{main}}.

Starting with the rational map $f$ constructed as above and applying Theorem \ref{self-grafting} successively, we obtain a sequence of rational maps $\{f_n\}$ such that $\deg f_n=3$ and $N(f_n)=n$ for $n\ge 1$.

More precisely, we know that $f_1=f$ has infinitely many periodic Jordan curves as buried Julia components and the Shishikura tree map $\tau_{f_1}$ has infinitely many repelling periodic cycles. Then by Theorem \ref{self-grafting}, we obtain the rational map $f_2$ with degree 3 such that $N(f_2)=N(f_1)+1=2$, and the Shishikura tree map $\tau_{f_2}$ has infinitely many repelling periodic cycles. By Theorem \ref{tree}, $f_2$ has infinitely many periodic Jordan curves as buried Julia components. So we could apply Theorem \ref{self-grafting} to $f_2$. Inductively, we get the sequence $\{f_n\}$.

Fix $n\ge 1$, for any integer $d>3$, applying the disc-annulus surgery in \cite{PT2}, we could obtain a rational map $g_n$ such that $\deg g_n=d$ and $N(g_n)=n$. The following is a detailed construction of $g_n$.

Let $(\UUU,\LLL)$ be a canonical decomposition of $f_n$. Let $U$ be a Q-type periodic component of $\UUU$ and $U_1$ be a non-periodic component of $f_n^{-1}(U)$. Take a quasi-disk $\Omega\Subset U_1\smm\PPP_{f_n}$ such that $f_n$ is injective on $\overline{\Omega}$. Pick another quasi-disk $\Delta\Subset\Omega$. Then there is a quasi-regular branched covering $G$ of $\cbar$ with $\deg G=d$ such that

(1) $G=f_n$ on $\cbar\smm\overline{\Omega}$,

(2) $G:\, \Delta\to\cbar\smm f_n(\Omega)$ is a holomorphic proper map with degree $d-3$, and

(3) $G:\, \Omega\smm\overline{\Delta}\to f_n(\Omega)$ is a quasi-regular branched covering with degree $d-2$.

Refer to Figure 9 for the construction of $G$. It is clear that the forward orbit of any point under $G$ passes through $\overline{\Omega}\smm\Delta$ at most once. Thus by Shishikura's principle (Lemma 1 in \cite{Shi1}), $G$ is quasi-conformally conjugated to a rational map $g_n$.

\begin{figure}[htbp]\centering
\includegraphics[width=12cm]{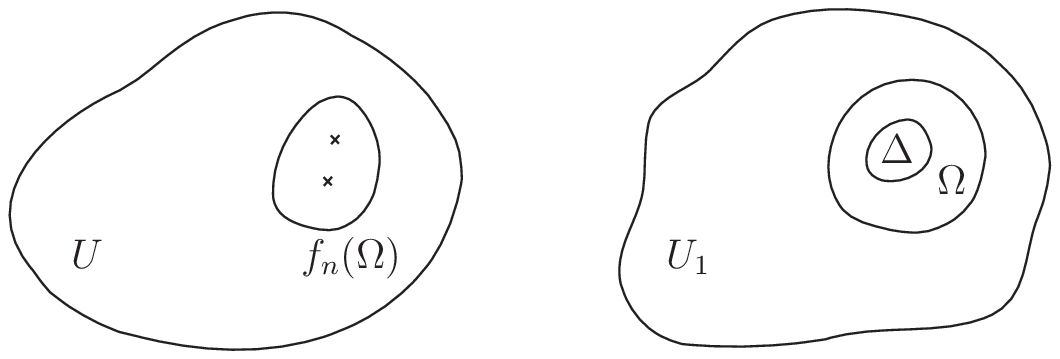}
\begin{center}{\sf Figure 9. The construction of $G$}
\end{center}
\end{figure}

Obviously, $(\UUU,\LLL)$ is still a canonical decomposition of $G$ and the Shishikura tree map of $G$ is the same as $\tau_{f_n}$. Thus $N(g_n)=N(f_n)=n$. \qed

\vspace{2mm}
\noindent{\it Acknowledgements}. The authors would like to express the sincere gratitude to the referees for valuable and helpful suggestions.

\noindent
Guizhen Cui \\
College of Mathematics and Statistics, Shenzhen University,\\
Shenzhen 518061, P. R. China;\\
HCMS and NCMIS, Academy of Mathematics and Systems Science, \\
Chinese Academy of Sciences, Beijing 100190, P. R. China.\\
gzcui@math.ac.cn

\vskip 0.24cm

\noindent
Wenjuan Peng \\
HLM and NCMIS, Academy of Mathematics and Systems Science, \\
Chinese Academy of Sciences, Beijing 100190, P. R. China.\\
wenjpeng@amss.ac.cn


\begin{thebibliography}{WI}

\bibitem{AC} M. Arfeux and G. Cui, \textit{Arbitrary large number of non trivial rescaling limits}, Arxiv: 1606.09574v1.


\bibitem{BH} B. Branner and J. Hubbard, \textit{The iteration of cubic polynomials, I, The global topology of parameter space}, Acta Math., 160 (1988), 143-206.

\bibitem{BCT} X. Buff, G. Cui and L. Tan, \textit{Teichm\"{u}ller spaces and holomorphic dynamics}, Handbook of Teichm\"{u}ller theory, Vol. IV, ed. Athanase Papadopoulos, Soci¨¦t¨¦ Math¨¦matique Europ¨¦enne (2014), 717-756.

\bibitem{CT} G. Cui and  L. Tan, \textit{A characterization of hyperbolic rational maps},  Invent. Math., 183 (2011), 451-516.

\bibitem{CPT} G. Cui, W. Peng and  L. Tan, \textit{Renormalization and wandering Jordan curves of rational maps}, Comm. Math. Phy., 344 (2016), 67-115.

\bibitem{DH1} A. Douady and J. H. Hubbard, \textit{A proof of Thurston's topological characterization of rational functions},  Acta Math., 171 (1993), 263-297.

\bibitem{DH2} A. Douady and J. H. Hubbard, \textit{\'{E}tude dynamique des polyn\^{o}mes complexes}, Partie I (Publications
Math\'{e}matiques d¡¯Orsay, 84), Universit\'{e} de Paris-Sud, Orsay, 1984.

\bibitem{DM} L. DeMarco and C. T. McMullen, \textit{Trees and the dynamics of polynomials},  Ann. Sci. \'{E}c. Norm. Sup\'{e}r.
(4), 41 (2008), 337-383.

\bibitem{G} S. Godillon, \textit{A family of rational maps with buried Julia components},  Ergodic Theory Dynam.
Systems, 35 (2015), 1846-1879.

\bibitem{JZ} Y. Jiang and G. Zhang, \textit{Combinatorial characterization of sub-hyperbolic rational maps}, Adv. Math., 221 (2009), 1990-2018.

\bibitem{KS} O. Kozlovski and S. van Strien, \textit{Local connectivity and quasi-conformal rigidity of non-renormalizable polynomials}, Proc. London Math. Soc.,  99 (2009), 275-296.

\bibitem{Mi} J. Milnor, \textit{Dynamics in One Complex Variable}, 3$^{rd}$ edition, Princeton University Press, 2006.

\bibitem{Mc1} C. T. McMullen, \textit{Automorphisms of rational maps}, In Holomorphic functions and moduli I, 31-60. Springer-Verlag, 1988.

\bibitem{Mc2} C. T. McMullen, \textit{Complex Dynamics and Renormalizations}, Ann. of Math. Stud., No. 135, Princeton University Press, 1994.

\bibitem{PT1} K. Pilgrim and L. Tan, \textit{Rational maps with disconnected Julia set}, Ast\'{e}risque 261 (2000), volume sp\'{e}cial en l'honneur d'A. Douady, 349-384.

\bibitem{PT2} K. Pilgrim and L. Tan, \textit{Disc-annulus surgery on rational maps}, a section in: B. Branner and N. Fagella, Quasiconformal surgery in holomorphic dynamics, Cambridge Stud. Adv. Math. 141, Cambridge University Press, Cambridge 2014, 267-282.

\bibitem{QY} W. Qiu and Y. Yin, \textit{Proof of the Branner-Hubbard conjecture on Cantor Julia sets}, Sci. China Ser. A,  52 (2009), 45-65.

\bibitem{Shi1} M. Shishikura, \textit{On the quasiconformal surgery of rational functions}, Ann. Sci. \'{E}c. Norm. Sup., 20 (1987), 1-29.

\bibitem{Shi2} M. Shishikura, \textit{Trees associated with the configuration of Herman rings}, Ergodic Theory Dynam. Systems,
9 (1989), 543-560.

\bibitem{Shi3}M. Shishikura, \textit{A new tree associated with Herman rings}, in proceedings of the conference "Complex dynamics and related fields," held in RIMS, Kyoto University, Surikaisekikenkyusho Kokyuroku, 1269 (2002), 75-92.

\bibitem{Th} W. Thurston, \textit{The combinatorics of iterated rational maps} (1985), published in: 'Complex dynamics: Families and Friends', ed. by D. Schleicher, A K Peters (2008), 1-108.

\end{thebibliography}
\end{document}